%% file: main.tex
\documentclass[a4paper, 12pt]{amsart}
\input{preamble}

\begin{document}

\begin{abstract}
    We introduce the class $LW_+$ of $L$-tauberian operators and its dual counterpart, the class $MW_-$ of $M$-cotauberian operators. We obtain sequential, perturbative, and algebraic characterizations of these classes. As a consequence, every $T\in LW_+(E,Y)$ is disjointly non-singular, and the converse implication holds if and only if every dispersed subspace of $E$ has equi-integrable unit ball. For $E=L_p$, the latter property depends on $p$. For $1\leq p<2$, equi-integrability of the unit ball is equivalent to dispersedness, and $L$-tauberian operators coincide with disjointly non-singular operators, while for $p\geq 2$ these equivalences fail. Moreover, $LW_+(L_p)$ is neither open nor stable under composition. Applying a result of Haydon, Odell, and Schlumprecht \cite{Haydon11}, we show that for $2<p<\infty$ every infinite dimensional dispersed subspace of $L_p$ contains an infinite dimensional subspace with equi-integrable unit ball. However, we give an example showing that this is not true for $p=2$.
\end{abstract}

\pagestyle{headings}
\markboth{\shortauthors}{\shorttitle}

\maketitle


\section{Introduction}

Disjointness is a device by which the order structure of a Banach lattice influences its linear geometry. The disjointly strictly singular operators are those operators which fail to be isomorphisms on subspaces generated by disjoint sequences. Contrarily, the disjointly non-singular operators ($\DNS$ operators, for short) are those operators which are strictly singular on no such subspace. They were introduced in \cite{Manuel21}, and have since been studied from quantitative, topological, perturbative, and dual viewpoints in \cite{Bilokopytov22,Bilokopytov25,Caldera26, ManuelQuantitative21,Manuel24}. 

\medskip

A closed subspace $M$ of a Banach lattice $E$ is called \emph{dispersed} if there is no normalized disjoint sequence $(x_n)\subset E$ such that $\textrm{dist}(x_n,M)\to 0$.
Dispersed subspaces have attracted interest through their connection with phase retrieval, a reconstruction problem arising in X-ray crystallography, astronomy, quantum mechanics, and speech recovery; see \cite{Freeman24} and the references therein. The kernel of every $\DNS$ operator is dispersed, and the quotient map associated with every dispersed subspace is $\DNS$. Thus, $\DNS$ operators provide an operator-theoretic counterpart of dispersed subspaces.

\medskip

There is another way of measuring how far a bounded set is from exhibiting disjoint behavior. A bounded subset of a Banach lattice $E$ is equi-integrable in $E$ when every disjoint sequence in its solid hull converges to zero in norm (cf. \cite[Definition 3.3.1]{Meyer-Nieberg:91}). For a closed subspace of a Banach lattice, having equi-integrable unit ball is stronger than being dispersed. 
For unit balls of subspaces of \(L_p\), this distinction disappears for $1\leq p<2$, but becomes significant for $2\leq p<\infty$. Understanding this change of behavior is one of the main purposes of the paper.

\medskip

We consider a class of operators  that arises by replacing relative weak compactness by equi-integrability in the definition of tauberian operators; see \cite{TauberianOperators,GonzalezOnieva90}. Let $E,F$ be Banach lattices and let $X,Y$ be Banach spaces. An operator $T:E\to Y$ is $L$-tauberian when every bounded set $C\subset E$ for which $T(C)$ is relatively compact is equi-integrable. Dually, an operator $S:X\to F$ is $M$-cotauberian when the adjoint $S^*:F^*\to X^*$ is $L$-tauberian. The algebraic characterizations we obtain explain the terminology: $T$ is $L$-tauberian precisely when $TA$ compact forces $A$ to be $L$-weakly compact, while $S$ is $M$-cotauberian precisely when $BS$ compact forces $B$ to be $M$-weakly compact. Thus, relative to the class $LW$, $L$-tauberian operators play the role that tauberian operators play relative to weakly compact operators. 
Also, when $E$ is an $L$-space the $L$-tauberian operators are exactly the tauberian operators, 
and the name $M$-cotauberian is dictated by duality. 

\medskip

The paper is organized as follows. Section~2 fixes our notation and recalls some facts about dispersed subspaces and $\DNS$ operators that will be used throughout. We record the sequential and perturbative characterizations of $\DNS$ operators and the relevant structure of dispersed subspaces of $L_p$. We collect some elementary properties of equi-integrable sets, introduce a measure of non-equi-integrability $\mathcal{NI}$, and obtain useful characterizations in function spaces. We also prove that every bounded equi-integrable subset of a Banach lattice is a Banach--Saks set (Theorem~\ref{thm: equi_Banach_Saks}), improving a result of \cite{Lesnik}.

\medskip

Section~3 introduces and studies $L$-tauberian and $M$-cotauberian operators. 
After establishing some sequential characterizations, we prove perturbative and algebraic characterizations parallel to those available for $\DNS$ and coDNS (with $\DNS$ adjoint) operators, and deduce that every $L$-tauberian operator is $\DNS$ and tauberian. We also characterize the closed-range case, examine $L$-tauberian projections and second adjoints, and isolate the order-continuity assumptions under which the classes are non-trivial.

\medskip

Section~4 considers the  case $E=L_p$, showing a different  behavior in the cases $1\leq p<2$, $p=2$, and $2<p<\infty$. 
The results obtained are summarized in the following result. 


\begin{theorem}\label{thm:binary}
    Given $1 \leq p < \infty$, the following statements are equivalent:
    \begin{enumerate}[(i)]
        \item $1 \leq p < 2$.
        \item A closed subspace $M$ of $L_p$ is dispersed if and only if $B_M$ is equi-integrable.
        \item An operator $T:L_p \to Y$ is $\DNS$ if and only if $T$ is $L$-tauberian.
        \item $LW_+(L_p, Y)$ is an open set of $\mathcal{B}(L_p, Y)$ for every Banach space $Y$.
        \item An operator $T: L_p \to L_p$ is $L$-tauberian if and only if for every bounded subset $C \on L_p$ with $T(C)$ equi-integrable, $C$ is equi-integrable.
        \item If $T,S \in LW_+(L_p)$, then $TS \in LW_+(L_p)$.
    \end{enumerate}
    Moreover,
    \begin{enumerate}[(i)]
        \item If $2 < p < \infty$, every infinite dimensional dispersed subspace $M$ of $L_p$ contains an infinite dimensional subspace $N$ such that $B_N$ is equi-integrable.
        \item If $p = 2$, there exists an infinite dimensional dispersed subspace $M$ of $L_2$ that contains no infinite dimensional subspace $N$ with $B_N$ equi-integrable.
    \end{enumerate}
\end{theorem}


\subsection*{Notation.}
All Banach spaces and Banach lattices are over the real scalar field. Throughout the paper, $X,Y,Z$ denote Banach spaces and $E,F$ denote Banach lattices. We write $E_+=\{x\in E:x\geq0\}$, and $B_X$ and $S_X$ for the closed unit ball and the unit sphere of $X$, respectively. If $(x_n)$ is a sequence in $X$, then $[x_n]$ denotes its closed linear span. 

A sequence is \emph{seminormalized} if it is bounded and bounded away from zero. We denote by $d(E)$ the set of disjoint sequences in $E\setminus\{0\}$ and by $dn(E)$ the set of normalized disjoint sequences in $E$. For a subset $A\subset E$, $\overline A$, $\con(A)$, $\acon(A)$, and $\sol(A)$ stand for its norm closure, convex hull, absolutely convex hull and solid hull, respectively; thus
\begin{equation*}
 \sol(A)=\{x\in E: |x|\leq |y|\text{ for some }y\in A\}.
\end{equation*}

\medskip

Operators are always linear and bounded, $\mathcal{B}(X,Y)$ denotes the space of all operators from $X$ into $Y$, and we write $\mathcal{B}(X)$ when $X=Y$. The identity operator is denoted by $I$. For $T\in\mathcal{B}(X,Y)$, we write $N(T)$ and $R(T)$ for its kernel and range, $T|_M$ for its restriction to a closed subspace $M\subset X$, and $T^*$ and $T^{**}$ for its first and second adjoints. Given a closed subspace $M$ of $X$, we denote by $J_M:M\to X$ the inclusion and by $Q_M:X\to X/M$ the quotient map. 
The annihilator of $M$ in $X^*$ is $M^\perp$.

\medskip

An operator $T\in\Lc(X,Y)$ is \emph{strictly singular} if there is no closed infinite-dimensional subspace $M$ of $X$ such that $T|_M$ is an isomorphism; the operator $T$ is \emph{upper semi-Fredholm} if $N(T)$ is finite-dimensional and $R(T)$ is closed, and it is \emph{lower semi-Fredholm} if $T^*$ is upper semi-Fredholm; equivalently, if $R(T)$ has finite codimension in $Y$. The operator $T$ is \emph{tauberian} if its second conjugate $T^{**}:X^{**}\to Y^{**}$ satisfies $T^{**-1}(Y) = X$, equivalently, if the induced operator $T^{\mathrm{co}}: X^{**}/X \to Y^{**}/Y$ is injective; and $T$ is called \emph{cotauberian} if $T^*$ is tauberian. We refer to \cite{TauberianL1,TauberianOperators,ManuelSaksman95,KaltonWilansky76} for further properties of tauberian and cotauberian operators and of the induced operator $T^{\mathrm{co}}$.

\medskip 

For a closed subspace $M$ of $X$ and $x\in X$, $d(x,M)=\inf_{y\in M}\Vert x-y\Vert$. If $M,N$ are closed subspaces of $X$, their one-sided gap and gap are
\begin{equation}
    \delta(M,N)=\sup_{x\in S_M}d(x,N),\qquad
 g(M,N)=\max\{\delta(M,N),\delta(N,M)\}.
\end{equation}

We write $L_p$ for $L_p[0,1]$, $1\leq p\leq \infty$. 
When we represent a Banach lattice as a K\"othe function space over a finite measure space $(\Omega,\Sigma,\mu)$, $\chi_A$ denotes the characteristic function of $A\in\Sigma$, and $\supp f$ is the essential support of a measurable function $f$. We write $A_n\downarrow\emptyset$ when $(A_n)$ is decreasing and $\mu(\bigcap_n A_n)=0$. 

\section{Preliminaries}

We use standard Banach lattice terminology and refer to \cite{Meyer-Nieberg:91} for the definitions and results used throughout. The norm of $E$ is \emph{order continuous} if $x_\alpha\downarrow0$ implies $\Vert x_\alpha\Vert\to0$. The order continuous part of $E$ is the closed ideal
\begin{equation*}
    E^a=\{x\in E: 0\leq x_n\leq |x|,\ x_n\downarrow0 \Longrightarrow \Vert x_n\Vert\to0\}.
\end{equation*}
Thus $E$ is order continuous precisely when $E=E^a$. A vector $e\in E_+$ is a \emph{weak unit} if $|x|\wedge e=0$ implies $x=0$. Every order continuous Banach lattice with a weak unit admits a representation as a K\"othe function space over a probability space $(\Omega,\Sigma,\mu)$; see \cite[Theorem~1.b.14]{LindenstraussTzafiri}. We fix such a representation whenever it is needed. In particular,
\begin{equation}
 L_\infty(\mu)\subset E\subset L_1(\mu),\qquad
 \Vert f\Vert_{L_1}\leq \Vert f\Vert_E\leq2\Vert f\Vert_{L_\infty}\quad(f\in L_\infty(\mu)),
\end{equation}
with $L_\infty(\mu)$ dense in $E$ and $E$ dense in $L_1(\mu)$.


Relative to this fixed representation, a closed subspace $M\subset E$ is called \emph{strongly embedded} if convergence in measure and norm convergence agree on $M$. Since norm convergence in $E$ implies convergence in measure, this means that every sequence $(x_n)\subset M$ which converges to zero in measure also converges to zero in norm.


\medskip

A sequence $(x_n)$ in $E$ is \emph{unbounded norm convergent} (or \emph{un-convergent}) to $x \in E$ if $(|x_n - x| \wedge u)$ converges in norm to $0$ for each $u \in E_+$~\cite{Troitsky17}. In this case we write $x_n \xrightarrow{un} x$. 


\subsection{Dispersed subspaces and disjointly non-singular operators.}{\label{sec: dispers}}

A closed subspace $M$ of a Banach lattice $E$ is \emph{dispersed} if no normalized disjoint sequence approaches it; equivalently, there is no $(x_n)\in dn(E)$ such that $d(x_n,M)\to0$. Thus dispersed subspaces contain no almost disjoint normalized sequence. This notion was introduced together with $\DNS$ operators in \cite{Manuel21} and has subsequently proved useful in the study of unbounded-norm convergence and phase retrieval; see \cite{Bilokopytov22,Freeman24,Manuel24}.

\begin{proposition}\label{prop: dispersed_strongly_embedded}
Let $E$ be an order continuous Banach lattice with a weak unit, and let $M$ be a closed subspace of $E$. Then $M$ is dispersed if and only if it is strongly embedded.
\end{proposition}

\begin{proof}
    By \cite[Theorem~4.5]{Bilokopytov22}, $M$ is dispersed if and only if the un-topology and the norm topology agree on $M$. In the fixed K\"othe representation, un-convergence of sequences agrees with convergence in measure by \cite[Corollary~2.12 and Theorem~4.6]{Troitsky17}. Hence the preceding condition is precisely strong embeddedness.
\end{proof}

An operator $T:E\to Y$ is \emph{disjointly strictly singular}, written $T\in\DSS(E,Y)$, if $T|_{[x_n]}$ is an isomorphism for no $(x_n)\in d(E)$. It is \emph{disjointly non-singular} \cite{Manuel21}, written $T\in\DNS(E,Y)$, if $T|_{[x_n]}$ is strictly singular for no $(x_n)\in d(E)$. Dually, an operator $S:X\to F$ is called \emph{codisjointly non-singular} \cite{Caldera26}, written $S\in\coDNS(X,F)$, if its adjoint $S^*:F^*\to X^*$ is disjointly non-singular. The following characterization combines the sequential and perturbative forms that are used later.

\begin{theorem}{\cite[Theorems~2.8 and 2.10]{Manuel24}, \cite[Theorem~5.3]{Bilokopytov22}}{\label{thm: DNS_caract}} For an operator $T: E \to Y$, the following statements are equivalent:
\begin{enumerate}[(i)]
    \item $T$ is $\DNS$.
    \item For every $(x_n) \in d(E)$, the restriction $T|_{[x_n]}$ is an upper semi-Fredholm operator.
    \item For every $(x_n) \in dn(E)$, $\liminf_{n \to \infty} \Vert Tx_n \Vert > 0$.
    \item $N(T+K)$ is dispersed for every compact operator $K: E \to Y$.
\end{enumerate}
Moreover, if $E$ is order continuous, the latter are equivalent to:
\begin{enumerate}[(i), start=5]
    \item There exists $r > 0$ such that for every $(x_n) \in dn(E)$, $\liminf_{n \to \infty} \Vert Tx_n \Vert > r$.
\end{enumerate}
In particular, $\DNS(E, Y)$ is an open subset of $\Lc(E,Y)$ whenever $E$ is order continuous.
\end{theorem}

\subsection{Equi-integrable sets} 

We now turn from dispersed subspaces to the stronger set-theoretic notion that underlies the definition of $L$-tauberian operators.

\begin{definition}{\cite{Meyer-Nieberg:91}}
    We say that a bounded subset $C$ of a Banach lattice $E$ is \textit{equi-integrable} (or \textit{$L$-weakly compact}) if $\Vert x_n \Vert \to 0$ for every disjoint sequence $(x_n) \on \text{sol}(C)$. We also define a measure of non-equi-integrability for a bounded subset $C\subset E$ as
    \begin{equation}
        \mathcal{NI}(C) = \sup\bigg\{\limsup_n \Vert x_n \Vert: (x_n) \on \text{sol}(C), \, \, (x_n) \,  \text{disjoint} \bigg\}.
    \end{equation}
\end{definition}

\begin{remark}
    Recall that a bounded subset $C \on E$ is said to be \textit{almost order bounded} if for every $\epsilon > 0$ there is $u \in E_+$ such that $C \on [-u, u] + \epsilon B_E$. It is easy to show (see for instance \cite[Proposition~3.6.2]{Meyer-Nieberg:91}) that in an order continuous Banach lattice $E$, a bounded subset $C$ is equi-integrable if and only if $C$ is almost order bounded. However the converse fails if we drop the order continuity, since every bounded subset of $C[0,1]$ is almost order bounded, but $C[0,1]^a = \{0\}$, and hence the only equi-integrable set is $\{0\}$. 
\end{remark}

Let us see some basic properties of equi-integrable sets.




\begin{proposition}{\cite[Section~3.6]{Meyer-Nieberg:91}}{\label{prop: prop_equi}}
    Let $C, D$ be bounded subsets of $E$, $\lambda \in \R$ and $x \in E$.
    \begin{enumerate}[(i)]
        \item  $\{x\}$ is equi-integrable if and only if $x \in E^a$.
        \item If $C$ is equi-integrable, then $C$ is relatively weakly compact and $C \on E^a$.
        \item If $C \on E^a$ and $C$ is relatively compact, then $C$ is equi-integrable.
        \item If $C \on D$ and $D$ is equi-integrable, then $C$ is equi-integrable.
        \item If $C$ is equi-integrable, then $\lambda C$ is equi-integrable.
        \item If $C, D$ are equi-integrable, then $C + D$, $C \cup D$ and $C \cap D$ are equi-integrable.
        \item $C$ is equi-integrable if and only if $\overline{C}$, $\con(C)$ or $\sol(C)$ is equi-integrable.
    \end{enumerate}
\end{proposition}

\begin{remark}{\label{rem: sum_equi}}
    Indeed, we can prove in a similar way that, given bounded subsets $C,D$ of a Banach lattice $E$,
    $\mathcal{NI}(C + D) \leq \mathcal{NI}(C) + \mathcal{NI}(D)$ and $\mathcal{NI}(C) \leq \sup_{x \in C} \Vert x \Vert$.
\end{remark}

The following standard characterization of equi-integrability will be used
repeatedly; we include the proof for completeness.

\begin{proposition}{\label{cor: equi_kothe}}
    Let $E$ be an order continuous Banach lattice with a weak unit, and let $C$ be a bounded subset of $E$. The following are equivalent:
    \begin{enumerate}[(i)]
        \item $C$ is equi-integrable.
        \item If $(A_n)$ is a disjoint family of measurable sets, then $\sup_{f \in C}\Vert f\chi_{A_n}\Vert \to 0$ as $n \to \infty$.
        \item If $A_n \downarrow \emptyset$, then $\sup_{f \in C}\Vert f\chi_{A_n}\Vert \to 0$ as $n \to \infty$.
        \item If $\mu(A_n) \to 0$, then $\sup_{f \in C}\Vert f\chi_{A_n}\Vert \to 0$ as $n \to \infty$.
        \item For every $\epsilon > 0$ there is $\delta > 0$ such that $\Vert f \chi_A \Vert < \epsilon$ for every $f \in C$ and every measurable set $A$ with $\mu(A) < \delta$.
        \item For every $\epsilon > 0$ there is $M > 0$ such that $\Vert f \chi_{\{|f| > M\}} \Vert < \epsilon$ for every $f \in C$.
    \end{enumerate}
\end{proposition}

\begin{proof}
    $(i) \Leftrightarrow (v)$ This can be found in \cite[Lemma~3.4]{Flores08}
    
    $(v) \Rightarrow (iv) \Rightarrow (iii)$ Immediate.

    $(iii) \Rightarrow (ii)$ Let $(A_n)$ be a disjoint family of measurable sets. If we consider the sets $B_n = \cup_{k=n}^{\infty} A_k$, since the measure space has finite measure it is easy to see that $B_n \downarrow 0$. Hence
    \begin{equation}
        0 \leq \sup_{f \in C}\Vert f\chi_{A_n}\Vert \leq \sup_{f \in C}\Vert f\chi_{B_n}\Vert \to 0.
    \end{equation}
    $(ii) \Rightarrow (v)$ The proof can be found in \cite[Lemma~2.2]{Lesnik}, since the argument still works when $E$ is an order continuous Köthe function space.

    $(v) \Rightarrow (vi)$ Consider an arbitrary $\epsilon > 0$ and let $K > 0$ such that $C \on K B_E$. Choose $\delta > 0$ such that  $\Vert f \chi_A \Vert < \epsilon$ for every $f \in C$ and every measurable set $A$ with $\mu(A) < \delta$. Take $M > K/\delta$. Then for every $f \in C$, $\mu(\{ |f| > M \}) \leq M^{-1}\Vert f \Vert \leq M^{-1}K < \delta$ by Markov's inequality, and hence $\sup_{f \in C} \Vert f \chi_{\{|f| > M\}}\Vert < \epsilon$.

    $(vi) \Rightarrow (v)$ Fix an arbitrary $\epsilon > 0$. Take $M > 0$ such that $\sup_{f \in C} \Vert f \chi_{\{|f| > M\}}\Vert < \epsilon/2$. Since $E$ is order continuous, there is $\delta > 0$ such that $\Vert \chi_A \Vert < \epsilon/(2M)$ whenever $\mu(A) < \delta$. Now, consider a measurable set $A$ with $\mu(A) < \delta$, and $f \in C$. Then
    \begin{equation}
        \Vert f \chi_A \Vert \leq \Vert f \chi_{A \cap \{|f| \leq  M \}} \Vert + \Vert f \chi_{A \cap \{|f| >  M \}} \Vert \leq M \Vert \chi_A \Vert + \Vert f \chi_{\{|f| > M \}}\Vert < \epsilon,
    \end{equation}
    as we wanted.
\end{proof}


For unit balls of subspaces, equi-integrability is always stronger than dispersedness. This elementary observation will later connect $L$-tauberian and $\DNS$ operators.

\begin{proposition}\label{prop: equi_ball_dispersed}
    Let $M$ be a closed subspace of a Banach lattice $E$. If $B_M$ is equi-integrable in $E$, then $M$ is dispersed. In particular, $M$ is reflexive.
\end{proposition}

\begin{proof}
    Suppose that $M$ is not dispersed. Then there are sequences $(x_n)\subset S_M$ and $(y_n)\in dn(E)$ such that $\Vert x_n-y_n\Vert\to0$. The sequence
    \begin{equation*}
     z_n=|x_n|\wedge |y_n|
    \end{equation*}
    is disjoint and contained in $\sol(B_M)$, so equi-integrability gives $\Vert z_n\Vert\to0$. On the other hand,
    \begin{equation*}
     1=\Vert y_n\Vert
     \leq \Vert |y_n|-z_n\Vert+\Vert z_n\Vert
     \leq \Vert y_n-x_n\Vert+\Vert z_n\Vert\longrightarrow0,
    \end{equation*}
    a contradiction. Thus $M$ is dispersed. Finally, every equi-integrable bounded set is relatively weakly compact; see \cite[Proposition~3.6.5]{Meyer-Nieberg:91}. Therefore $B_M$ is weakly compact and $M$ is reflexive.
\end{proof}

The Rademacher span supplies the basic infinite-dimensional example that will recur in Section~4.

\begin{example}\label{ex: Rademacher}
    Let $(r_n)$ be the Rademacher system in $L_p$, $1\leq p<\infty$. Then $B_{[r_n]}$ is equi-integrable in $L_p$, and in particular $[r_n]$ is a dispersed subspace of $L_p$.
\end{example}

\begin{proof}
    Choose $q>p$. By Khintchine's inequality (cf. \cite[Theorem~6.2.3]{KaltonAlbiac}) there is a constant $C_{p,q}>0$ such that $\Vert f\Vert_q\leq C_{p,q}\Vert f\Vert_p$ for every $f\in [r_n]$. Hence, for every measurable set $A\subset[0,1]$ and every $f\in B_{[r_n]}$, H\"older's inequality gives
    \begin{equation}
     \Vert f\chi_A\Vert_p\leq \mu(A)^{1/p-1/q}\Vert f\Vert_q
     \leq C_{p,q}\mu(A)^{1/p-1/q}.
    \end{equation}
    Hence $B_{[r_n]}$ is equi-integrable by Proposition~\ref{cor: equi_kothe}. Finally, by Proposition~\ref{prop: equi_ball_dispersed} we conclude that $[r_n]$ is a dispersed subspace of $L_p$.
\end{proof}

We have seen that bounded equi-integrable sets are relatively weakly compact. We now show that they are, in fact, Banach--Saks sets. Recall that a subset $C$ of a Banach space $X$ is a \textit{Banach-Saks set} if for every sequence $(x_n) \on C$ there exist $x \in X$ and a subsequence $(x_{n_k})$ such that every further subsequence $(x_{n_{k_l}})$ has means converging in norm to $x$, i.e. $ (\sum_{l=1}^n x_{n_{k_l}})/n$ converges in norm to $x$. The proof of the following result adapts \cite[Theorem~3.1]{Lesnik}, where the conclusion is established for Banach function lattices with the Fatou property and containing $L_1\cap L_\infty$.

\begin{proposition}{\label{thm: equi_Banach_Saks}}
    Every bounded equi-integrable subset of a Banach lattice is Banach-Saks.
\end{proposition}

\begin{proof}
    First, we assume that $E$ is  order continuous  with a weak unit. Thus, we may see $E$ as an ideal of $L_1(\mu)$, where $\mu$ is a probability measure and $\Vert f \Vert_{1} \leq \Vert f \Vert \leq 2\Vert f \Vert_{\infty}$ for every $f \in L_{\infty}(\mu)$. Suppose that $C$ is a bounded equi-integrable subset of $E$ and let $(f_n) \on C$ be an arbitrary sequence. By Komlós theorem (cf. \cite{Komlos}) there is a subsequence $(f_{n_k})$ of $(f_n)$ and a function $f \in L_1(\mu)$ such that for any further subsequence $(f_{n_{k_l}})$ of $(f_{n_k})$, $\frac{1}{n}\sum_{l=1}^n f_{n_{k_l}} \to f $ $\mu$-a.e. Now, let $(f_{n_{k_l}})$ be an arbitrary subsequence of $(f_{n_k})$ and consider $g_n = \frac{1}{n}\sum_{l=1}^n f_{n_{k_l}}$. We will see that $(g_n)$ is a Cauchy sequence.
    
    Let $\epsilon > 0$. Since $C$ is equi-integrable, we know by Proposition~\ref{prop: prop_equi} that $(g_n) \on \con(C)$ is equi-integrable. Therefore, by Proposition~\ref{cor: equi_kothe} there is $0 < \delta < 1$ such that $\Vert g_n \chi_A \Vert < \epsilon$ for every $n \geq 1$ and every measurable set $A \on \Omega$ with $\mu(A) < \delta$. Now, since $g_n \to f$ $\mu$-a.e., by Egorov's theorem there is a measurable set $B \on \Omega$ such that $\mu(B) < \delta$ and $g_n$ converges uniformly to $f$ on $\Omega \setminus B$. Take $n_0 \in \N$ such that 
    \begin{equation}
        \Vert (g_n - f) \chi_{\Omega \setminus B} \Vert_{\infty} \leq \dfrac{\epsilon}{\Vert \chi_{\Omega \setminus B} \Vert},
    \end{equation}
    for every $n \geq n_0$. Then, for every $n, m \geq n_0$,
    \begin{equation}
        \begin{aligned}
            \Vert g_n - g_m \Vert & \leq \Vert (g_n - g_m)\chi_{B} \Vert + \Vert (g_n - g_m)\chi_{\Omega \setminus B} \Vert \\
            & \leq 2 \sup_k \Vert g_k\chi_{B} \Vert + 2\sup_{k \geq n_0} \Vert (g_n - f)\chi_{\Omega \setminus B} \Vert < 4\epsilon.
        \end{aligned}
    \end{equation}
    Therefore $(g_n)$ is a Cauchy sequence, and by the completeness of $E$ there exist $g \in E$ such that $\Vert g_n - g \Vert \to 0$. In particular $g_n \xrightarrow{\mu} g$, so $f = g$ and we conclude that $f \in E$ and $g_n \to f$ in $E$. Hence $C$ is Banach-Saks.

    Now, suppose that $E$ is a general Banach lattice. Since every equi-integrable subset lies in $E^a$ (Proposition~\ref{prop: prop_equi}), we may assume without loss of generality that $E$ is order continuous. Suppose by contradiction that there is a bounded equi-integrable subset $C$ of $E$ which is not a Banach-Saks set. Then, we can find a bounded equi-integrable sequence $(x_n)$ that is not Banach-Saks. Using that every closed sublattice of an order continuous Banach lattice is order continuous, and replacing $E$ by the closed sublattice generated by $(x_n)$, we may assume that $E$ is separable. It follows in particular that $E$ has a weak unit, yielding a contradiction with the first case proved.
\end{proof}

Recall that a subset of a Banach space is super-weakly compact if all of its ultraproducts are relatively weakly compact. Such  sets are Banach-Saks. The preceding theorem
therefore suggests the following natural strengthening.

\begin{question}
    Is every equi-integrable set super-weakly compact?
\end{question}

\section{L-tauberian operators and M-cotauberian operators}

We first recall $L$-weakly and $M$-weakly compact operators, which motivate the classes studied in this section.

\begin{definition} (\cite[Definition~3.6.9]{Meyer-Nieberg:91})
    \begin{enumerate}[(i)]      
         \item An operator $S:X \to F$ is \textit{$L$-weakly compact} if $S(B_X)$ is equi-integrable.  
        \item An operator $T:E \to Y$ is \textit{$M$-weakly compact} if $\|Tx_n\|\to 0$ for every disjoint sequence $(x_n)\subset B_E$. 
    \end{enumerate}
\end{definition}

We denote by $LW(X,F)$ and $MW(E,Y)$ the corresponding sets of operators, and by $W(X,Y)$ the weakly compact operators.

\begin{proposition}{\label{prop: prop_L-Mweak}} \cite[Section 3.6]{Meyer-Nieberg:91} Let $S:X \to F$ and $T:E \to Y$. 
    \begin{enumerate}[(i)]
        \item $S\in LW$ if and only if $S^*\in MW$. 
        \item $T\in MW$ if and only if $T^* \in LW$. 
        \item $LW(X,F)$ and $MW(E,Y)$ are closed subspaces of $\Lc(X,F)$ and $\Lc(E,Y)$. 
        \item $LW(X,F)\subset W(X,F)$ and $MW(E,Y)\subset W(E,Y)$. 
        \item If $F$ is order continuous, then $MW(E,F)\subset$ $LW(E,F)$
    \end{enumerate}
\end{proposition}
    
\begin{corollary}{\label{cor: inc_quot}}
    Let $M$ be a closed subspace of $E$. 
    \begin{enumerate}[(i)]
        \item $B_M$ is equi-integrable (in $E$) if and only if the inclusion $J_M:M\to E$ is $L$-weakly compact. 
        \item $B_{M^\perp}$ is equi-integrable (in $E^*$) if and only if the quotient map $Q_M:E\to E/M$ is $M$-weakly compact. 
    \end{enumerate}
\end{corollary}
    
\begin{definition}
    An operator $T:E \to Y$ is \textit{$L$-tauberian} if every bounded subset $C$ of $E$ with  $T(C)$ relatively compact is equi-integrable. Similarly, an operator $S:X \to F$ is \textit{$M$-cotauberian} if $S^*$ is $L$-tauberian.
\end{definition}

We write $LW_+(E,Y)$ for the class of $L$-tauberian operators from $E$ into $Y$, and $MW_-(X,F)$ for the class of $M$-cotauberian operators from $X$ into $F$. The terminology reflects the fact that $L$-tauberian operators on $L$-spaces are precisely the classical tauberian operators, while $M$-cotauberian operators on $M$-spaces are precisely the cotauberian operators (see Proposition~\ref{prop: tauberian_L_space} and Theorem~\ref{thm: L_taub_DNS_taub} below).

\begin{proposition}\label{prop: tauberian_L_space}
Let $E$ be an $L$-space and $F$ be an $M$-space. Then every tauberian operator $T:E\to Y$ is $L$-tauberian, and every cotauberian operator $S: X \to F$ is $M$-cotauberian.
\end{proposition}

\begin{proof}
Let $T:E\to Y$ be a tauberian operator and $C\subset E$ be bounded and suppose that $T(C)$ is relatively compact. Since $T$ is tauberian, $C$ is relatively weakly compact by \cite[Corollary~2.2.5]{TauberianOperators}. On an $L$-space, relatively weakly compact sets are $L$-weakly compact; see \cite[Section~3.6]{Meyer-Nieberg:91}. Hence $C$ is equi-integrable, and therefore $T$ is $L$-tauberian.

The dual statement follows from the fact that $F^*$ is an $L$-space whenever $F$ is an $M$-space (see \cite[Proposition~1.4.7]{Meyer-Nieberg:91}).
\end{proof}

The first point to settle is when these classes are non-empty. The answer also explains why order continuity is the natural standing hypothesis in the remainder of the section.

\begin{proposition}{\label{prop: non_trivial_L_taub}}
    Let $E,F$ be Banach lattices.
    \begin{enumerate}[(i)]
         \item If $LW_+(E, Y) \neq \emptyset$ for some Banach space $Y$, then $E$ is order continuous.
        \item If $E$ is order continuous, then every upper semi-Fredholm operator $T: E \to Y$ is $L$-tauberian.
    \end{enumerate}
    Similarly,
    \begin{enumerate}[(i), start = 3]
         \item If $MW_-(X, F) \neq \emptyset$ for some Banach space $X$, then $F^*$ is order continuous.
        \item If $F^*$ is order continuous, then every lower semi-Fredholm operator $T: X \to F$ is $M$-cotauberian.
    \end{enumerate}
\end{proposition}

\begin{proof}
    $(i)$ Suppose that there is an $L$-tauberian operator $T: E \to Y$, and let $x \in E$. Since $\{Tx\}$ is compact, it follows that $\{x\}$ is equi-integrable, i.e. $x \in E^a$.

    $(ii)$ Suppose that $E$ is order continuous. Since every relatively compact subset is equi-integrable (Proposition~\ref{prop: prop_equi}), this follows from the fact that an operator $T: X \to Y$ between Banach spaces is upper semi-Fredholm if and only if for every bounded subset $C \on X$ for which $T(C)$ is relatively compact, $C$ is relatively compact \cite[p.~400]{GonzalezOnieva90}.

    $(iii)-(iv)$ These ones follow by duality from $(i)$ and $(ii)$ respectively.
\end{proof}

By Proposition~\ref{prop: non_trivial_L_taub}, the existence of an $L$-tauberian operator $T:E\to Y$ already forces $E$ to be order continuous. Nevertheless, we retain this hypothesis explicitly in the statements below for clarity and to keep the dual formulations symmetric. We begin with operators having closed range, for which the $L$-tauberian and $M$-cotauberian properties admit a simple description in terms of the kernel and the quotient by the range.

\begin{proposition}{\label{prop: L_tau_closed_range}}
    Let $E,F$ be Banach lattices such that $E, F^*$ are order continuous. 
    \begin{enumerate}[(i)]
        \item An operator $T:E\to Y$ with closed range is $L$-tauberian if and only if $B_{N(T)}$ is equi-integrable. 
        \item An operator $S:X\to F$ with closed range is $M$-cotauberian if and only if $Q_{\overline{R(S)}}$ is $M$-weakly compact.
    \end{enumerate}
\end{proposition}

\begin{proof}
    $(i)$ Note that $T$ is $L$-tauberian if and only if so is $Q_{N(T)}$. Indeed, in this case $\hat T$ is an isomorphism. So, given a bounded subset $C$ of $E$, $T(C)$ is relatively compact if and only if so is $Q_{N(T)}(C)$. 
    
    It remains to show that for a closed subspace $N$ of $E$, $Q_N$ is $L$-tauberian if and only if $B_N$ is equi-integrable. The direct implication is immediate. For the converse, let $C\subset B_E$ with $Q_N(C)$ relatively compact. For each $0<\varepsilon<1$ there exists $x_1,\ldots,x_n\in C$ such that $Q_N(C)\subset \cup_{i=1}^n B(\hat x_i,\varepsilon)$, where $\hat x_i= Q_N(x_i)$. Note that $Q_N$ takes the open unit ball of $E$ onto the open unit ball of $E/N$. Therefore, 
    $$
    C\subset 3B_N + \cup_{i=1}^n B(x_i,\varepsilon)=  3B_N + \{x_1,\ldots,x_n\}+ B(0,\varepsilon). 
    $$
    By Remark~\ref{rem: sum_equi} it follows that
    \begin{equation}
        \mathcal{NI}(C)\leq \mathcal{NI}(3B_N)+ \mathcal{NI}(\{x_1,\ldots,x_n\})+ \mathcal{NI}(B(0, \varepsilon)) = \mathcal{NI}(B(0, \varepsilon))\leq \varepsilon.
    \end{equation}
    Since $\epsilon > 0$ is arbitrary,  $C$ is equi-integrable and the proof is complete. 

    $(ii)$ The direct implication is immediate. For the converse, note that $S: X \to F$ has closed range if and only if $S^*: F \to X$ has closed range. Moreover, $N(S^*) = R(S)^\perp$, so $B_{N(S^*)}$ is equi-integrable by Corollary~\ref{cor: inc_quot}, and the result follows from $(i)$.
\end{proof}

When the range is not assumed to be closed, a similar characterization can be obtained by considering all compact perturbations of the operator. To establish it, we need the following two lemmas. 

\begin{lemma}{\label{lemma: equiint_order_continuous}}
    Let $(x_n)$ be a bounded sequence on an order continuous Banach lattice $E$. If $(x_n)$ is not equi-integrable, there is a subsequence $(x_{n_k})$ of $(x_n)$ such that no subsequence of $(x_{n_k}-x_{n_{k-1}})$ is equi-integrable.
\end{lemma}

\begin{proof}
    If $(x_n)$ is not equi-integrable, by passing to a subsequence if necessary, we may assume that there is a disjoint sequence $(y_n)$ such that $|y_n| \leq |x_n|$ and $\Vert y_n \Vert \geq c > 0$ for every $n \geq 1$. We will construct a subsequence of $(x_n)$ as follows.

    First, take $n_1 = 1$, $n_0 = 0$, $x_0 = 0$ and $\tilde{y}_1 = y_1$. Clearly $|\tilde{y}_1| \leq |x_{n_1} - x_{n_0}|$ and $\Vert \tilde{y}_1 \Vert \geq c \geq c/2$. Now, since $x_{n_1} \in E^a$, we know that $\lim_n \Vert |y_n| \wedge |x_{n_1}| \Vert = 0$. Hence, we can choose $n_2 > n_1$ such that $\Vert |y_{n_2}| \wedge |x_{n_1}| \Vert < c/2$. Take $\tilde{y}_2 = |y_{n_2}| - |y_{n_2}|\wedge |x_{n_1}|$. Clearly $\Vert \tilde{y}_2 \Vert \geq c/2$, and
    \begin{equation}
        0 \leq \tilde{y}_2 = (|y_{n_2}| - |x_{n_1}|)_+ \leq (|x_{n_2}| - |x_{n_1}|)_+ \leq  ||x_{n_2}| - |x_{n_1}|| \leq |x_{n_2} - x_{n_1}|.
    \end{equation}
    Since $x_{n_2} \in E^a$, we can choose $n_3 > n_2$ such that $\Vert |y_{n_3}| \wedge |x_{n_2}| \Vert < c/2$. Continuing this process we obtain a subsequence $(x_{n_k})$ of $(x_n)$, and a sequence $(\tilde{y}_k)$ such that $\tilde{y}_k = |y_{n_k}| - |y_{n_k}|\wedge |x_{n_{k-1}}|$, $0 \leq \tilde{y}_k \leq |x_{n_k} - x_{n_{k-1}}|$ and $\Vert \tilde{y}_k \Vert \geq c/2$. Moreover $\tilde{y}_k \leq |y_{n_k}|$, so $(\tilde{y}_k)$ is disjoint. Thus, no subsequence of $(x_{n_k}-x_{n_{k-1}})$ is equi-integrable, just as we wanted.
\end{proof}

\begin{lemma}{\label{lemma: equiint_order_continuous_conv}}
    Let $(x_n)$ be a bounded sequence on an order continuous Banach lattice $E$. If $(x_n)$ is not equi-integrable and $x \in E$,   there is a subsequence $(x_{n_k})$ such that no subsequence of $(x_{n_k}-x)$ is equi-integrable.
\end{lemma}

\begin{proof}
    If $(x_n)$ is not equi-integrable, by passing to a subsequence if necessary, we may assume that there is a disjoint sequence $(y_n)$ such that $|y_n| \leq |x_n|$ and $\Vert y_n \Vert \geq c > 0$ for every $n \geq 1$. Since $E$ is order continuous, we know that $\lim_n \Vert |y_n| \wedge |x| \Vert = 0$. Hence, there is $n_0 \in \N$ such that $\Vert |y_{n}| \wedge |x| \Vert < c/2$ for every $n \geq n_0$. Take $\tilde{y}_n = |y_n| - |y_n|\wedge |x|$. Clearly $(\tilde{y}_n)$ is disjoint, $\Vert \tilde{y}_n \Vert \geq c/2$, and
    \begin{equation}
        0 \leq \tilde{y}_n = (|y_{n}| - |x|)_+ \leq (|x_n| - |x|)_+ \leq  ||x_{n}| - |x|| \leq |x_{n} - x|,
    \end{equation}
    for every $n \geq n_0$. Thus no subsequence of $(x_{n}-x)_{n \geq n_0}$ is equi-integrable.
\end{proof}

We will need the following criterion for the existence of basic subsequences. 

\begin{theorem} \cite[Theorem~1.5.6]{KaltonAlbiac}{\label{thm: eber_smul}}
    Let $S\subset X$ bounded with $0\notin\overline{S}$. Then $S$ contains no basic subsequence if and only if $\overline{S}^w$ is weakly compact and $0\notin\overline{S}^w$. 
\end{theorem}

Combining this selection principle with the two preceding lemmas yields a sequential
test for $L$-tauberian operators. Its basic-sequence formulations will be the main tool in the
perturbative argument.

\begin{theorem}{\label{thm: carac_L_taub_seq}}
    If $E$ is order continuous, for an operator $T: E \to Y$ the following assertions are equivalent:
    \begin{enumerate}[(i)]
        \item $T$ is $L$-tauberian.
        \item For every bounded sequence $(x_n) \on E$ such that $Tx_n$ converges, $(x_n)$ is equi-integrable.
         \item For every semi-normalized basic sequence $(x_n) \on E$ with no equi-integrable subsequences, $\limsup_{n \to \infty} \Vert Tx_n \Vert > 0$.
    \end{enumerate}
\end{theorem}

\begin{proof}
    Clearly $(i) \Rightarrow (ii) \Rightarrow (iii)$, so it suffices to prove that $(iii) \Rightarrow (i)$ Suppose that $T$ is not $L$-tauberian, i.e. there is a bounded subset $C$ such that $T(C)$ is relatively compact and $C$ is not equi-integrable. Then, there is a bounded sequence $(x_n)$ such that $(Tx_n)$ converges and $(x_n)$ has no equi-integrable subsequence. Note that $(x_n)$ has no convergent subsequences either, and take $S=\{x_n :n\in\N\}$.

    First, suppose that $S$ is not relatively weakly compact. By the Eberlein–Šmulian theorem \cite[Theorem~1.6.3]{KaltonAlbiac}, we may pass to a subsequence and assume that $(x_n)$ has no weakly null subsequences, so there is a functional $x^* \in E^*$ such that $\lambda \coloneq \limsup_n x^*(x_n) > 0$. Thus, we may select a subsequence $(x_{n_k})$ of $(x_n)$ such that $ 0 < \lambda/2 < x^*(x_{n_k}) < 3\lambda/2$ and $\alpha_k \coloneq x^*(x_{n_k}) \to \lambda$ as $k \to \infty$. Therefore, if we consider the sequence $y_k = \alpha_k^{-1} x_{n_k}$, the set $S' = \{y_n :n\in\N\}$ is not relatively weakly compact, $(y_n)$ is a seminormalized sequence, $(y_n)$ has no equi-integrable subsequences, $(Ty_n)$ converges,  and there is $x^* \in E^*$ such that $x^*(y_n) = 1$ for every $n \geq 1$. Thus by Theorem~\ref{thm: eber_smul} we may assume that the sequence $(y_n)$ is basic, and by Lemma~\ref{lemma: equiint_order_continuous} and \cite[Lemma~2.44]{TauberianOperators}, we can pass to a subsequence $(y_{n_k})$ of $(y_n)$ such that $z_n = y_{n_k} - y_{n_{k-1}}$ is basic and has no equi-integrable subsequences. This implies that $(z_n)$ is seminormalized, and clearly $\Vert Tz_n \Vert \to 0$, which is a contradiction.

    Now, suppose that $S$ is relatively weakly compact, i.e. by passing to a subsequence we may assume that $x_{n} \overset{w}{\longrightarrow} x$. Then, by Lemma~\ref{lemma: equiint_order_continuous_conv} we may pass to a further subsequence such that $(x_n - x)$ has no equi-integrable subsequence. Clearly $(x_n-x)$ has no convergent subsequences, so by Theorem~\ref{thm: eber_smul} we can finally pass again to a subsequence such that $(x_n - x)$ is basic. Hence $(x_n - x)$ is a seminormalized weakly null basic sequence, and since $Tx_n$ converges, it follows that $\Vert T(x_n - x) \Vert \to 0$. This concludes the proof.
\end{proof}

For the dual statement we shall use the following weak$^*$ selection lemma, which produces a bounded biorthogonal family in the predual.

\begin{lemma}{\cite[Lemma~3.1.19]{TauberianOperators}}{\label{lemma: basic_weak^*_conv}} Let $(g_n)$ be a seminormalized sequence in a dual space $Y^*$ such that $0$ is a weak$^*$ cluster point of $\{g_n : n \in \N\}$. Then $(g_n)$ has a subsequence $(g_{n_k})$ for which there exists a bounded sequence $(y_k)$ in $Y$ so that $g_{n_i}(y_j) = \delta_{ij}$ for all $i, j \in \N$. 
\end{lemma}

We can now characterize both classes through compact perturbations. The form of the result deliberately parallels the perturbative description of $\DNS$ and $\coDNS$ operators recalled in Section~2 and developed in \cite{Caldera26}.

\begin{theorem}[Perturbative characterization for $LW_+$ and $MW_-$]\label{thm:pert_L-W+}    
    Let $E,F$ be Banach lattices such that $E, F^*$ are order continuous.
    \begin{enumerate}[(i)]
        \item An operator $T: E \to Y$ is $L$-tauberian if and only if $B_{N(T+K)}$ is equi-integrable for every compact operator $K: E \to Y$.
        \item An operator $S: X \to F$ is $M$-cotauberian if and only if $Q_{\overline{R(S+K)}}$ is $M$-weakly compact for every compact operator $K: X \to F$.
    \end{enumerate}       
\end{theorem}

\begin{proof}
    $(i)$ Suppose that $T:E \to Y$ is $L$-tauberian  and $K: E \to Y$ is a compact operator. Clearly $T|_{N(T+K)} = - K|_{N(T+K)}$, so $T(B_{N(T+K)})$ is relatively compact, and hence $B_{N(T+K)}$ is equi-integrable. 

    Now, suppose that $T$ is not $L$-tauberian. By Theorem~\ref{thm: carac_L_taub_seq} we know that there is a seminormalized basic sequence $(x_n)$ such that $\Vert Tx_n \Vert \to 0$ and $(x_n)$ contains no equi-integrable subsequence, and thus there exists a bounded sequence $(x^*_n)$ in $E^*$ such that $x^*_i(x_j)=\delta_{ij}$. By passing to a subsequence, we may assume that $\Vert Tx_n \Vert \leq 2^{-n}$ and $(x_n)$ is still not equi-integrable. Then, $Kx= -\sum_{n=1}^\infty x^*_{n}(x) Tx_n$ defines a compact operator $K:E\to Y$ with $(x_n)\subset N(T+K)$. Hence, $B_{N(T+K)}$ is not equi-integrable.

    $(ii)$ Suppose that $S^*$ is $L$-tauberian and $K: X \to F$ is a compact operator. Then $K^*$ is compact, and by the first part of the proof $B_{N(S^*+ K^*)} = B_{R(S+K)^\perp}$ is equi-integrable. Thus $Q_{\overline{R(S+K)}}$ is $M$-weakly compact by Corollary~\ref{cor: inc_quot}.



    Conversely, if $S: X \to F$ is not $M$-cotauberian (i.e. $S^*: F^* \to X^*$ is not $L$-tauberian) then there exists a sequence $(x_n^*)$ in $B_{F^*}$ such that $S^*x_n^* \to z^* \in X^*$ and $(x_n^*)$ has no equi-integrable subsequences. By the Banach-Alaoglu theorem, $B_{F^*}$ is weak$^*$-compact, so $\{x_n^*: n \in \N\}$ has a weak$^*$-cluster point $x^* \in F^*$.  Clearly  $S^*x^* = z^*$ and $0$ is a weak$^*$-cluster point of $\{y_n^*: n \in \N\}$, where $y_n^* \coloneq x_n^*-x^*$. Moreover, after discarding finitely many terms there is $c>0$ such that $\Vert y_n^*\Vert\geq c$ for every $n \geq 1$ (otherwise some subsequence of $(x_n^*)$ would converge in norm to $x^*$ and hence be equi-integrable by Proposition~\ref{prop: prop_equi}). Thus, applying Lemma~\ref{lemma: basic_weak^*_conv} we can find a bounded sequence $(y_n)$ in $F$ and a subsequence of $(y_n^*)$ (still denoted $(y_n^*)$) such that $y_i^*(y_j) = \delta_{ij}$ for every $i,j \geq 1$.

    By Lemma~\ref{lemma: equiint_order_continuous_conv} we may pass to a further subsequence such that $(y_n^*)$ is bounded away from $0$, $(y_n^*)$ is not equi-integrable, $\Vert y_n \Vert\Vert S^*y_n^* \Vert \leq 2^{-n}$ and $y_i^*(y_j) = \delta_{ij}$ for every $i,j \geq 1$. Then, $Kx= -\sum_{n=1}^\infty y^*_{n}(Sx) y_n$ defines a compact operator $K:X\to F$, and since $K^*z^*= -\sum_{n=1}^\infty z^*(y_n) S^*y_n^*$, it follows that $(y_n^*) \on N(S^*+ K^*)$. Thus, $B_{N(S^*+ K^*)} = B_{R(S+K)^\perp}$ is not equi-integrable, so $Q_{\overline{R(S+K)}}$ is not $M$-weakly compact by Corollary~\ref{cor: inc_quot}.
\end{proof}

\begin{theorem}\label{thm: L_taub_DNS_taub}
    Every $L$-tauberian operator is $\DNS$ and tauberian. Similarly, every $M$-cotauberian operator is $\coDNS$ and cotauberian.
\end{theorem}

\begin{proof}
Let $T:E\to Y$ be $L$-tauberian. Proposition~\ref{prop: non_trivial_L_taub} implies that $E$ is order continuous. By Theorem~\ref{thm:pert_L-W+}, $B_{N(T+K)}$ is equi-integrable for every compact operator $K:E\to Y$. Proposition~\ref{prop: equi_ball_dispersed} therefore shows that $N(T+K)$ is dispersed for every such $K$. The perturbative characterization for DNS operators (see Theorem~\ref{thm: DNS_caract}) now gives $T\in\DNS(E,Y)$.

Moreover, whenever $C\subset E$ is bounded and $T(C)$ is relatively compact, the set $C$ is equi-integrable and hence relatively weakly compact. The standard characterization of tauberian operators in \cite[Corollary~2.2.5]{TauberianOperators} then shows that $T$ is tauberian.

The last assertion follows easily by duality.
\end{proof}

\begin{corollary}\label{cor: L_space_taub_equivalence}
    If $E$ is an $L$-space, an operator $T:E\to Y$ is $L$-tauberian if and only if it is tauberian. Similarly, if $F$ is an $M$-space, an operator $T:X\to F$ is $M$-cotauberian if and only if it is cotauberian.
\end{corollary}

\begin{proof}
    The direct implication follows from Theorem~\ref{thm: L_taub_DNS_taub}, and the converse is Proposition~\ref{prop: tauberian_L_space}.
\end{proof}

The preceding perturbative characterization has the following algebraic counterpart, relating $LW_+$ to $LW$ and $MW_-$ to $MW$.

\begin{proposition}[Algebraic characterization of $LW_+$ and $MW_-$]\label{prop:alg-char-L-taub} 
        Let $E,F$ be Banach lattices such that $E, F^*$ are order continuous.
    \begin{enumerate}[(i)]
        \item An operator  $T:E\to Y$ is $L$-tauberian if and only if $A:X\to E$ is $L$-weakly compact whenever $TA$ is compact.
        \item An operator $S:X\to F$ is $M$-cotauberian if and only if $B:F\to Z$ is $M$-weakly compact whenever $BS$ is compact.
    \end{enumerate}
\end{proposition}

\begin{proof}
    $(i)$ "$\Rightarrow$" If $TA$ is compact, then $TA(B_X)$ is relatively compact. Hence $A(B_X)$ is equi-integrable and so $A$ is $L$-weakly compact. 

    "$\Leftarrow$" If $T$ is not $L$-tauberian, we know by Theorem~\ref{thm:pert_L-W+} that there exists a compact operator $K:E\to Y$ such that $B_{N(T+K)}$ is not equi-integrable. Thus the inclusion $J_{N(T+K)}: N(T+K)\to E$ is not $L$-weakly compact, but $TJ_{N(T+K)}$ is compact. 

    $(ii)$ "$\Rightarrow$" Since $S^*$ is $L$-tauberian, if $BS$ is compact then so is $S^*B^*$. Thus, by $(i)$ we know that $B^*$ is $L$-weakly compact, hence $B$ is $M$-weakly compact by Proposition~\ref{prop: prop_L-Mweak}.

    "$\Leftarrow$" If $S$ is not $M$-cotauberian then by Theorem~\ref{thm:pert_L-W+} there exists a compact operator $K:X\to F$ such that $Q_{\overline{R(S+K)}}$ is not $M$-weakly compact. However $Q_{\overline{R(S+K)}} S = -Q_{\overline{R(S+K)}} K$ is compact.
\end{proof}

\begin{remark}

    $(i)$ The use of $L$-weak compactness in Proposition~\ref{prop:alg-char-L-taub} (i) is essential. Indeed, suppose one defined an operator $T:X\to Y$ to be $M$-tauberian if, for every Banach lattice $E$ and every operator $A:E\to X$, compactness of $TA$ implied that $A$ was $M$-weakly compact. Then no operator with nonzero domain is $M$-tauberian. To see this, choose $0\neq x\in X$ and consider the rank-one operator $A:\ell_1\to X$ given by $Ae_n=x$ for every $n\in\mathbb N$. Then $TA$ is compact, whereas $A$ is not $M$-weakly compact.

    $(ii)$ For projections, the characterization of $L$-tauberian operators can be strengthened. If $E$ is order continuous and $P:E\to E$ is a projection, then $P$ is $L$-tauberian if and only if every bounded subset $C\subset E$ such that $P(C)$ is equi-integrable is itself equi-integrable. This follows easily from the decomposition $C\subset P(C)+(I-P)(C)$ and the equi-integrability of $B_{N(P)}$. The corresponding statement fails for general operators; see Corollary~\ref{cor: equi_equi} below.
\end{remark}

We finish the abstract part of the theory by comparing an operator with its second adjoint. The first observation ensures that equi-integrability is preserved when a sublattice is viewed inside a larger lattice.

\begin{lemma}{\label{lemma: sublatt_equiint}}
    Let $F$ be a sublattice of $E$, and let $C$ be a bounded subset of $F$. If $C$ is equi-integrable on $E$, then $C$ is equi-integrable on $F$.
\end{lemma}

\begin{proposition}{\label{prop: second_conj_L_taub}}
    Let $E$ be a Banach lattice such that $E^{**}$ is order continuous, and $T: E \to Y$. If $T^{**}: E^{**} \to Y^{**}$ is $L$-tauberian, then $T$ is $L$-tauberian.
\end{proposition}

\begin{proof}
    This is a consequence of Lemma~\ref{lemma: sublatt_equiint}, and the fact that $E$ is a sublattice of $E^{**}$ (cf. \cite[Proposition~1.4.5]{Meyer-Nieberg:91}).
\end{proof}

\begin{proposition}
    Suppose that $E^{**}$ is order continuous and $T: E \to Y$ is an operator with closed range. Then $T$ is $L$-tauberian if and only if $T^{**}$ is $L$-tauberian.
\end{proposition}

\begin{proof}
    Suppose that $T$ is $L$-tauberian. Note that $R(T)$ is closed if and only if $R(T^{**})$ is closed, so by Proposition~\ref{prop: L_tau_closed_range} it suffices to prove that $B_{N(T^{**})}$ is equi-integrable. Since $B_{N(T)}$ is equi-integrable, by Proposition~\ref{prop: prop_equi} $B_{N(T)}$ is relatively weakly compact. Thus $N(T)$ is reflexive and $N(T^{**}) = N(T)^{\perp \perp} = J_E N(T)$, where $J_E: E \to E^{**}$ is the canonical embedding. Now suppose that $B_{N(T^{**})}$ is not equi-integrable, i.e. there is a disjoint sequence $(y_n) \on E^{**}$ and a sequence $(x_n) \on B_{N(T)}$ such that $|y_n| \leq |J_Ex_n|$ and $\Vert y_n \Vert \geq c > 0$ for every $n \geq 1$. Since $E$ is also order continuous, $E$ is an ideal in $E^{**}$ (cf. \cite[Theorem~1.17]{Wnuk}), and there is a disjoint sequence $(z_n) \on E$ such that $y_n = J_Ez_n$. Thus $|z_n| \leq |x_n|$, so $\Vert z_n \Vert \to 0$ since $B_{N(T)}$ is equi-integrable. However $J_E$ is an isometry, so $\Vert y_n \Vert \to 0$, which is a contradiction.
\end{proof}

\begin{question}
    Suppose that $E^{**}$ is order continuous (equivalently $E^{**}$ contains no copy of $c_0$, see \cite{Wnuk}) and let $T: E \to Y$ be $L$-tauberian. Is $T^{**}: E^{**} \to Y^{**}$ $L$-tauberian? 
\end{question}

In this direction, it was proved (see \cite[Proposition~4.6]{Manuel24}) that if $E$ contains no copy of $c_0$ and $T:E\to Y$ is $\DNS$, then $T^{**}:E^{**}\to Y^{**}$ is also $\DNS$. Consequently, if $E^{**}$ is order continuous and $T:E\to Y$ is $L$-tauberian, we know by Theorem~\ref{thm: L_taub_DNS_taub} that $T^{**}$ is at least $\DNS$.

\section{$LW_+$ operators and subspaces with equi-integrable unit ball in $L_p$}\label{sect:equi-int}

In this section we study the closed subspaces of $L_p$ which are dispersed. We show that the properties for $1\leq p <2$ are different from those for $2\leq p<\infty$. 
\medskip

\noindent\textbf{The case $1 \leq p < 2$.}
For $p<2$, equi-integrability of a bounded absolutely convex set can be detected by the finite-dimensional $\ell_p$ structures that the set contains. This is the key input for identifying $L$-tauberian and $\DNS$ operators. The implication $(i) \Rightarrow (iii)$ in Theorem \ref{prop: equi_p} is a consequence of Dor's disjointification theorem \cite[Theorem~B]{Dor75}, while $(ii)\Rightarrow(i)$ follows from a minor adaptation of the proof of \cite[Theorem~2]{Astashkin14}

\begin{theorem}{\cite[Lemma~1]{JohnsonSchechtman08}}{\label{prop: equi_p}}
    Let $1\leq p < 2$ and let $C$ be a bounded absolutely convex subset of $L_p$. The following are equivalent:
    \begin{enumerate}[(i)]
        \item $C$ is equi-integrable.
        \item $C$ contains no sequence equivalent to the canonical basis of $\ell_p$.
        \item For every $K > 0$ there is $n \in \N$ so that no length $n$ sequence in $C$ is $K$-equivalent to the unit vector basis of $\ell_p^n$.
    \end{enumerate}
\end{theorem}

\begin{remark}
    The cited lemma is originally stated for $1\leq p\neq2<\infty$, with the authors asserting that the case $p>2$ follows from the results of Kadec--Pe\l{}czy\'nski \cite{Kadec}. However, this equivalence fails for $p>2$; see the explicit counterexample in Proposition~\ref{prop: count_disp_nonequi}.
\end{remark}

Applying this characterization to unit balls yields a criterion for dispersedness and, consequently, identifies $L$-tauberian operators on $L_p$ with $\DNS$ operators.

\begin{corollary}
    Let $1 \leq p < 2$. 
    \begin{enumerate}[(i)]
        \item A closed subspace $M \on L_p$ is dispersed if and only if $B_M$ is equi-integrable.
        \item An operator $T: L_p \to Y$ is $\DNS$ if and only if $T$ is $L$-tauberian.
    \end{enumerate}
    In particular, $LW_+(L_p, Y)$ is an open subset of  $\mathcal{B}(L_p, Y)$ for every Banach space $Y$.
\end{corollary}

\begin{proof}
    $(i)$ This follows from Proposition~\ref{prop: dispersed_strongly_embedded}, Theorem~\ref{prop: equi_p} and \cite[Proposition~3.4]{Manuel21}.

    $(ii)$ This is an immediate consequence of $(i)$ and the perturbative characterization of DNS operators (Theorem~\ref{thm: DNS_caract}) and $L$-tauberian operators (Theorem~\ref{thm:pert_L-W+}).

    Finally, since $LW_+(L_p, Y) = \DNS(L_p, Y)$,  Theorem~\ref{thm: DNS_caract} implies that $LW_+(L_p, Y)$ is an open subset of $\mathcal{B}(L_p, Y)$.
\end{proof}

The characterization of dispersed subspaces by equi-integrability of their unit balls also yields a stability result for one-sided gap perturbations. To simplify the proof, given an operator $T: L_p \to Y$ we will denote, as in \cite{Manuel21},
\begin{equation}
    \beta(T) \coloneq \inf\big\{ \liminf_{n\to \infty} \Vert Tx_n \Vert: (x_n)\in \text{dn}(L_p) \big\}.
\end{equation}

\begin{corollary}
    Let $1 \leq p < 2$ and $M, N$ be closed subspaces of $L_p$. If $M$ is dispersed and $\delta(N, M) < 1$, then $N$ is dispersed.
\end{corollary}

\begin{proof}
    Suppose that $M,N$ are closed subspaces of $L_p$ such that $M$ is dispersed and $\delta(N, M) < 1$. First, note that $\beta(Q_M) = 1$ by \cite[Proposition~3.12]{Manuel21}. Thus, for every normalized disjoint sequence $(x_n) \on L_p$, $\lim_{n \to \infty} d(x_n, M) = 1$. On the other hand, by \cite[Lemma IV.2.2]{Kato} we know that $(1+\delta(N,M))d(x,N)\geq d(x,M)-\delta(N,M)$ for every $x \in L_p$. Thus, it follows that 
    \begin{equation}
        \liminf_{n\to \infty} d(x_n, N) \geq \dfrac{\beta(Q_M)-\delta(N,M)}{ 1+ \delta(N,M)} = \dfrac{1-\delta(N,M)}{ 1+ \delta(N,M)} > 0
    \end{equation}
    for every normalized disjoint sequence $(x_n) \on L_p$. Hence $N$ is dispersed.
\end{proof}

\begin{remark}
    This result may be compared with the classical stability of several isomorphic properties under one-sided gap perturbations. If $\delta(N,M)<1$, then finite-dimensionality, superreflexivity, $B$-convexity, and the property of containing no copy of $\ell_1$ pass from $M$ to $N$. Separability, reflexivity, and quasireflexivity also pass from $M$ to $N$ under the stronger assumption $\delta(N,M)<1/2$; see \cite[Examples~2.3 and Corollary~2.11]{AAG:92}.
\end{remark}

For operators on $L_p$ the same equivalence admits a useful strengthening: relative compactness of the image may be replaced by equi-integrability.

\begin{corollary}{\label{cor: semi_L_p}}
    Let $1 \leq p < 2$ and $T: L_p \to L_p$. The following are equivalent:
    \begin{enumerate}[(i)]
        \item $T$ is $L$-tauberian.
        \item For every bounded subset $C$ of $L_p$ with $T(C)$ equi-integrable, $C$ is equi-integrable.
        \item For every bounded subset $C$ of $L_p$ such that $T(C)$ contains no sequence equivalent to the basis of $\ell_p$, then $C$ contains no sequence equivalent to the basis of $\ell_p$.
        \item $A: X \to L_p$ is $L$-weakly compact whenever $TA$ is $L$-weakly compact.
    \end{enumerate}
    In particular, if $T,S \in LW_+(L_p)$, then $TS\in LW_+(L_p)$.
\end{corollary}

\begin{proof}
    $(i) \Rightarrow (iii)$ Suppose that there exists a bounded subset $C$ of $L_p$ such that $T(C)$  contains no sequence equivalent to the basis of $\ell_p$, but there is a sequence $(f_n) \on C$ equivalent to the basis of $\ell_p$. Since $T$ is $\DNS$ by Theorem~\ref{thm: L_taub_DNS_taub}, we know by \cite[Proposition~3.8]{Manuel21} that $T|_{[f_n]}$ is upper semi-Fredholm. Hence $N(T|_{[f_n]})$ is finite-dimensional, and
    \cite[Lemma~2.2]{Manuel21} gives $n_0 \in \N$ such that $N(T|_{[f_n]})\cap[f_n:n\geq n_0]=\{0\}$. Since $N(T|_{[f_n]})+[f_n:n\geq n_0]$ and $R(T|_{[f_n]})$ are closed, $T$ is bounded below on $[f_n:n\geq n_0]$. Consequently, $(Tf_n)_{n\geq n_0}$
    is equivalent to the canonical basis of $\ell_p$, which yields a contradiction.

    $(iii) \Rightarrow (ii)$ Let $C$ be a bounded subset of $L_p$ such that $T(C)$ is equi-integrable. First, suppose that $C$ is absolutely convex. Then $T(C)$ is absolutely convex too, and hence $C$ is equi-integrable by Theorem~\ref{prop: equi_p}. Now, if $C$ is a general bounded subset of $L_p$, it is easy to check that $T(\acon(C)) = \acon(T(C)) = \con(T(C)\cup (-T(C)))$. Therefore, by Proposition~\ref{prop: prop_equi} it follows that $T(\acon(C))$ is equi-integrable, and hence, by the first case, $\acon(C)$ is equi-integrable. Since $C \on \acon(C)$, we conclude by Proposition~\ref{prop: prop_equi} that $C$ is equi-integrable too.
    
    $(ii) \Rightarrow (iv)$ If $TA$ is $L$-weakly compact, then $TA(B_X)$ is equi-integrable. Hence $A(B_X)$ is equi-integrable and $A$ is $L$-weakly compact.
    
    $(iv) \Rightarrow (i)$ This follows from the algebraic characterization of $L$-tauberian operators (Proposition~\ref{prop:alg-char-L-taub}) and the fact that every relatively compact subset in an order continuous Banach lattice is equi-integrable (see Proposition~\ref{prop: prop_equi}).

    The final assertion follows easily from $(ii)$.
\end{proof}

\noindent\textbf{The case $2 \leq p < \infty$.} 
The picture changes at $p=2$. We first separate equi-integrable unit balls from dispersed subspaces and then translate that separation into failures of openness and stability under composition.

The next result shows that being dispersed is strictly weaker than having equi-integrable unit ball for $p\geq2$.

\begin{proposition}{\label{prop: count_disp_nonequi}}
    Let $2 \leq p < \infty$. Then there is a dispersed subspace $M$ of $L_p$ such that $B_M$ is not equi-integrable.
\end{proposition}

\begin{proof}
    Let $(x_n) \on L_p$ be a disjoint normalized sequence such that $\supp x_n \on (1/2,1)$ for every $n \geq 1$, and let $(r_n)$ be the Rademacher system. Since $L_p[0,1/2]$ is lattice isomorphic to $L_p$, we may assume that $\supp r_n \on [0,1/2]$ for every $n \geq 1$. Consider the subspace $M=[x_n+r_n]$. Since $\mu(\supp x_n) \to 0$, we have
    \begin{equation}
        \mathcal{NI}(B_M) \geq \limsup_{n \to \infty} \dfrac{\Vert (x_n + r_n)\chi_{\supp x_n}\Vert_p}{\Vert x_n + r_n\Vert_p} \geq \limsup_{n \to \infty} \dfrac{\Vert x_n\Vert_p}{\Vert x_n \Vert_p + \Vert r_n\Vert_p}  = \dfrac{1}{2} > 0,
    \end{equation}
    and thus $B_M$ is not equi-integrable. Now, suppose that $M$ fails to be dispersed, i.e. there is a sequence $(z_n) \on S_M$ and $(\tilde{z}_n) \in dn(L_p)$ such that $\Vert z_n - \tilde{z}_n \Vert \to 0$. In particular, $z_n \xrightarrow{\mu}0$, and 
    \begin{equation}
        z_n = \sum_{k =1}^{\infty} a_k^n (x_k+r_k)= \sum_{k =1}^{\infty} a_k^nx_k + \sum_{k =1}^{\infty} a_k^n r_k = \tilde{x}_n + \tilde{y}_n,
    \end{equation}
    where $(\tilde{y}_n) \in [r_n]$ and $(\tilde{x}_n) \in [x_n]$. Then, by Khintchine's inequality (cf. \cite[Theorem~6.2.3]{KaltonAlbiac}), there is a constant $C > 0$ such that
    \begin{equation}
        \Vert \tilde{x}_n \Vert_p = \bigg \Vert \sum_{k =1}^{\infty} a_k^nx_k \bigg \Vert_p = \bigg(\sum_{k =1}^{\infty} |a_k^n|^p \bigg)^{1/p} \leq \bigg(\sum_{k =1}^{\infty} |a_k^n|^2 \bigg)^{1/2} \leq C\bigg \Vert \sum_{k =1}^{\infty} a_k^nr_k \bigg \Vert_p = C\Vert \tilde{y}_n \Vert_p
    \end{equation}
    Hence $\Vert \tilde{y}_n \Vert_p \geq (C+1)^{-1}\Vert z_n \Vert_p = (C+1)^{-1}$. On the other hand, note that $|\tilde{y}_n| \leq |z_n|$ and thus $\tilde{y}_n \xrightarrow{\mu}0$. Since $L_p$ is order continuous, we conclude by \cite[Theorem~3.2]{Troitsky17} that $[r_n]$ contains an almost disjoint sequence. This yields a contradiction, since $[r_n]$ is a dispersed subspace (see Example~\ref{ex: Rademacher}).
\end{proof}


The distinction between dispersedness and having equi-integrable ball has a further consequence: the set of $L$-tauberian operators can be non-open. The following result relies essentially on \cite[Proposition~3]{Astashkin23}.

\begin{proposition}{\label{prop: sub_ex_open}}
    Let $2 \leq p < \infty$, $\lambda \geq 0$. Consider an arbitrary sequence $(x_n) \in dn(L_p)$ with $\supp(x_n) \on [0,1/2]$, and a copy of the rademacher system $(r_n)$ with $\supp(r_n) \on [1/2, 1]$. Take the sequence $z_n = r_n + \lambda x_n$, and the closed subspace $M_{\lambda} = [z_n]$. Then,
    \begin{enumerate}[(i)]
        \item $B_{M_\lambda}$ is equi-integrable if and only if $\lambda = 0$.
        \item $M_{\lambda}$ is isomorphic to $\ell_2$ and complemented.
        \item $g(M_{\lambda}, M_{0}) \leq \lambda$.
    \end{enumerate}
\end{proposition}

\begin{proof}
    $(i)$ The sufficiency follows from Example~\ref{ex: Rademacher}. Now, suppose that $\lambda > 0$, and note that for every $n \geq 1$,
    \begin{equation}
         \dfrac{\Vert z_n \chi_{\supp(x_n)} \Vert_p}{\Vert z_n \Vert_p} = \dfrac{\lambda\Vert x_n \Vert_p}{(\Vert r_n \Vert_p^p + \lambda^p\Vert x_n \Vert_p^p)^{1/p}} = \dfrac{\lambda}{(1+\lambda^p)^{1/p}}.
    \end{equation}
    Since $\mu(\supp(x_n)) \to 0$, we conclude by Proposition~\ref{cor: equi_kothe} that $B_{M_{\lambda}}$ is not equi-integrable.

    $(ii)$ Consider an arbitrary $\sum_{n=1}^{\infty} a_n z_n \in M_{\lambda}$. Then, by Khintchine's inequality,
    \begin{equation}{\label{ec: ineq_con_rad}}
        \begin{aligned}
            \bigg \Vert \sum_{n=1}^{\infty} a_n z_n \bigg \Vert_p & = \bigg \Vert \sum_{n=1}^{\infty} a_n r_n + \lambda\sum_{n=1}^{\infty} a_n x_n \bigg \Vert_p  \leq \bigg \Vert \sum_{n=1}^{\infty} a_n r_n  \bigg \Vert_p + \lambda\bigg \Vert\sum_{n=1}^{\infty} a_n x_n \bigg \Vert_p  \\
            & \leq C \bigg( \sum_{n=1}^{\infty} |a_n|^2 \bigg)^{1/2} + \lambda \bigg( \sum_{n=1}^{\infty} |a_n|^p\bigg)^{1/p}  \leq (C + \lambda) \bigg( \sum_{n=1}^{\infty} |a_n|^2 \bigg)^{1/2}  \\
            & \leq (C + \lambda)\bigg \Vert \bigg|\sum_{n=1}^{\infty} a_n r_n \bigg|  \bigg \Vert_2 \leq (C + \lambda)\bigg \Vert \bigg|\sum_{n=1}^{\infty} a_n r_n \bigg| + \lambda\bigg|\sum_{n=1}^{\infty} a_n x_n \bigg| \bigg \Vert_2 \\
            & = (C + \lambda)\bigg \Vert \sum_{n=1}^{\infty} a_n z_n \bigg \Vert_2
        \end{aligned}
    \end{equation}
    Since the converse inequality always holds, we conclude that $M_{\lambda}$ is isomorphic to $\ell_2$. In the $p=2$ case $M_{\lambda}$ is trivially complemented, and for $p > 2$, this follows from \cite[Theorem~6.4.8]{KaltonAlbiac} and the fact that $M_{\lambda}$ is isomorphic to $\ell_2$.

    $(iii)$ Consider an arbitrary $(a_n) \in \ell_2$. Then
    \begin{equation}
        \bigg\Vert \sum_{n=1}^{\infty} a_n z_n  - \sum_{n=1}^{\infty} a_n r_n  \bigg\Vert_p  =  \lambda\bigg \Vert \sum_{n=1}^{\infty} a_n x_n \bigg \Vert_p = \lambda \Vert (a_n) \Vert_{\ell_p},
    \end{equation}
    Moreover, applying Khintchine's inequality and arguing as in \eqref{ec: ineq_con_rad} we have
    \begin{equation}
        \Vert (a_n) \Vert_{\ell_2} \leq \bigg\Vert \sum_{n=1}^{\infty} a_n z_n \bigg \Vert_p \leq (C + \lambda) \Vert (a_n) \Vert_{\ell_2} \textrm{ and }
        \Vert (a_n) \Vert_{\ell_2} \leq \bigg\Vert \sum_{n=1}^{\infty} a_n r_n \bigg \Vert_p \leq C \Vert (a_n) \Vert_{\ell_2}.
    \end{equation}
    Now, let $x = \sum_{n=1}^{\infty} a_n z_n \in S_{M_{\lambda}}$, and consider $y \coloneq \sum_{n=1}^{\infty} a_n r_n \in M_{0}$. Then
    \begin{equation}
        d(x, M_0) \leq \bigg\Vert x - \dfrac{y}{\Vert x \Vert_p} \bigg\Vert_p = \bigg\Vert \dfrac{x}{\Vert x \Vert_p} - \dfrac{y}{\Vert x \Vert_p} \bigg\Vert_p = \dfrac{\Vert x - y \Vert_p}{\Vert x \Vert_p} \leq \lambda \dfrac{\Vert (a_n) \Vert_{\ell_p}}{\Vert (a_n) \Vert_{\ell_2}} \leq \lambda
    \end{equation}
    Taking the supremum over all possible $x \in S_{M_{\lambda}}$, it follows that $\delta(M_{\lambda}, M_0) \leq \lambda$. Similarly, it is easy to see that $\delta(M_{0}, M_{\lambda}) \leq \lambda$, and hence $g(M_{\lambda}, M_0) \leq \lambda$.
\end{proof}

\begin{corollary}{\label{cor: seq_sub_not_open}}
    Let $2 \leq p < \infty$. There is a sequence of complemented closed subspaces $(M_n)$ of $L_p$ with $B_{M_n}$ not equi-integrable, and a closed subspace $M_0$ of $L_p$ with $B_{M_0}$ equi-integrable, such that $g(M_n,M_0) \to 0$.
\end{corollary}

\begin{proof}
    It suffices to consider the subspaces $M_{1/n}$, defined in Proposition~\ref{prop: sub_ex_open}.
\end{proof}

To pass from nearby complemented subspaces to nearby projections, we use two standard gap estimates. Recall that a closed subspace $M$ of a Banach space $X$ is complemented if and only if $X=M\oplus N$ for some closed subspace $N$ of $X$.

\begin{lemma}{\cite[Lemma~IV. 4.28]{Kato}}{\label{lemma: proj_gap}}
    Let  $M, N$ be closed subspaces of a Banach space $X$ such that $X = M \oplus N$. Then there is $\epsilon > 0$ such that $X = M' \oplus N$ for every closed subspace $M'$ with $g(M',M) < \epsilon$. 
\end{lemma}

\begin{lemma}\label{lemma: conv_rang_kern}
    Let $X=M\oplus N$, and let $(M_n)$ be closed subspaces such that $X=M_n\oplus N$ for every $n \geq 1$. Denote by $P$ and $P_n$ the projections
    onto $M$ and $M_n$, respectively, with kernel $N$. If
    $\delta(M_n,M)\to0$, then $\Vert P_n-P\Vert\to0$.
\end{lemma}

\begin{proof}
    For $u\in M_n$,
    \begin{equation}
        \Vert(I-P)u\Vert
        \leq \Vert I-P\Vert\,d(u,M)
        \leq \Vert I-P\Vert\,\delta(M_n,M)\Vert u\Vert.
    \end{equation}
    Consequently,
    \begin{equation}
        \Vert Pu\Vert
        \geq \Vert u\Vert-\Vert(I-P)u\Vert
        \geq \bigl(1-\Vert I-P\Vert\,\delta(M_n,M)\bigr)\Vert u\Vert.
    \end{equation}
    Note that, for sufficiently large $n$, $1-\Vert I-P\Vert\,\delta(M_n,M) > 0$.
    Since $P$ and $P_n$ have the same kernel, $P=PP_n$. Thus, for $x\in X$,
    \begin{equation}
        \begin{aligned}
            \Vert P_nx\Vert
            &\leq
            \frac{\Vert PP_nx\Vert}
                 {1-\Vert I-P\Vert\,\delta(M_n,M)} =\frac{\Vert Px\Vert}
                    {1-\Vert I-P\Vert\,\delta(M_n,M)} \leq
            \frac{\Vert P\Vert}
                 {1-\Vert I-P\Vert\,\delta(M_n,M)}\Vert x\Vert.
        \end{aligned}
    \end{equation}
    Finally,
    \begin{equation}
        \begin{aligned}
            \Vert(P_n-P)x\Vert =\Vert(I-P)P_nx\Vert \leq \Vert I-P\Vert\,\delta(M_n,M)\Vert P_nx\Vert \leq \frac{\Vert I-P\Vert\,\Vert P\Vert\,\delta(M_n,M)}
                 {1-\Vert I-P\Vert\,\delta(M_n,M)}\Vert x\Vert.
        \end{aligned}
    \end{equation}
    Hence $\Vert P_n-P\Vert\to0$.
\end{proof}


\begin{corollary}{\label{cor: seq_not_open}}
    For $2\leq p < \infty$, $LW_+(L_p)$ is not an open subset of $\mathcal{B}(L_p)$. 
\end{corollary}

\begin{proof}
    Let $2 \leq p < \infty$ arbitrary. By Corollary~\ref{cor: seq_sub_not_open} there is a sequence of complemented closed subspaces $(M_n)$ of $L_p$ with $B_{M_n}$ not equi-integrable for every $n \in \N$, $B_{M_0}$ equi-integrable  and $g(M_n, M_0) \to 0$. Therefore, by Lemma~\ref{lemma: proj_gap} we may assume that for every $n \geq 1$ the complement of  $M_n$ is $N_0$, the complement of $M_0$.
    
    Let $P_n, P: L_p \to L_p$ be the corresponding projections onto $M_n$ and $M_0$ with kernel $N_0$, and define the operators $T_n = I - P_n$ and $T = I - P$ for every $n \in \N$. Clearly  $N(T_n) = M_n$ and $N(T) = M_0$. Since $T_n$, $T$ all have closed range, $B_{M_{n}}$ fails to be equi-integrable and $B_{M_0}$ is equi-integrable, we know that $T_n$ is a sequence of non-$L$-tauberian operators and $T$ is $L$-tauberian. On the other hand, $g(M_n, M_0) \to 0$ and $(M_n), M$ have all the same complement. Thus we conclude by Lemma~\ref{lemma: conv_rang_kern} that $\Vert P_n - P \Vert \to 0$, and hence $\Vert T_n - T \Vert \to 0$, as we wanted.
\end{proof}


Corollary \ref{cor: seq_not_open} already shows that $LW_+(L_p)$ differs substantially from $\DNS(L_p)$ when $p\geq2$. The next construction gives a stronger distinction: even composing with an isomorphism can destroy $L$-tauberianity.

\begin{proposition}{\label{prop: L-taub_L_2}}
    For every $2 \leq p < \infty$ there exist an onto isomorphism $S\in \Lc(L_p)$ and an operator $T\in LW_+(L_p)$ such that $TS \notin LW_+(L_p)$. In particular, $LW_+(L_p)$ is not stable under composition for  $2 \leq p < \infty$.
\end{proposition}

\begin{proof}
    Consider a  sequence $(x_n) \in dn(L_p)$ with  $\supp x_n \on [0,1/2]$, and take the Rademacher system $(r_n)$ so that $\supp r_n \on [1/2,1]$. Since $[x_n]$ is isometric to $\ell_p$ and $[r_n]$ is isomorphic to $\ell_2$, we can find a bounded operator $A: [r_n] \to [x_n]$ such that $Ar_n = x_n$. Moreover $[r_n]$ is complemented in $L_p$ (see the proof of \cite[Proposition~6.4.2]{KaltonAlbiac}), so there is a projection $P : L_p \to L_p$ onto $[r_n]$. Indeed, since $L_p = L_p(0,1/2) \oplus_p L_p(1/2,1)$, $\supp r_n \on [1/2,1]$ and $\supp x_n \on [0,1/2]$, we may take the projection such that $Px_n = 0$ for every $n \geq 1$, and thus $PA = 0$. Consider $T=I-P$ and $S = I-AP$. Note that
    \begin{equation}
        (I+AP)S = I-AP +AP -APAP = I.
    \end{equation}
    Similarly $S(I+AP) = I$, so $S$ is an isomorphism. Moreover $N(T) = N(I-P) = [r_n]$ and $T$ has closed range. Since $B_{[r_n]}$ is equi-integrable (see Example~\ref{ex: Rademacher}), we know by Proposition~\ref{prop: L_tau_closed_range} that $T$ is $L$-tauberian. Finally, consider the subset $C = \{r_n +x_n\}_n \on L_p$. Clearly $C$ is bounded, and since $\Vert (x_n+r_n)\chi_{\supp x_n} \Vert_p = \Vert x_n \Vert_p = 1$, it follows that $C$ is not equi-integrable. Moreover,
    \begin{equation}
        S(x_n+r_n) = x_n+r_n-AP(x_n+r_n) = x_n+r_n-Ar_n = r_n,
    \end{equation}
    and thus $TS(x_n+r_n)=(I-P)r_n=0$. Hence $TS(C) = \{0\}$, which is compact, even though $C$ is not equi-integrable. Therefore $TS$ is not $L$-tauberian.
\end{proof}

\begin{corollary}{\label{cor: equi_equi}}
    Let $2 \leq p < \infty$. Then there is a surjective isomorphism $S\in \Lc(L_p)$ and a bounded subset $C \on L_p$ such that $S(C)$ is equi-integrable and $C$ is not.
\end{corollary}

\begin{proof}
    It suffices to take as $S$ the isomorphism in Proposition~\ref{prop: L-taub_L_2}, and the subset $C = \{r_n +x_n\}_n \on L_p$ with $(r_n)$ and $(x_n)$ as in its proof.
\end{proof}




For $2<p<\infty$, a dispersed subspace $M$ of $L_p$ can have non-equi-integrable unit ball. However, a result of Haydon, Odell, and Schlumprecht \cite{Haydon11}  gives the following fact: 

\begin{theorem}{\label{thm: disp_cont_equiball}}
    Let $2 < p < \infty$. Then every infinite dimensional dispersed subspace of $L_p$ contains an infinite dimensional subspace with equi-integrable unit ball.
\end{theorem}

\begin{proof}
    This is a consequence of \cite[Lemma~5.4]{Haydon11} and the fact that, for $p > 2$,  an infinite dimensional subspace of $L_p$ is dispersed if and only if it is isomorphic to $\ell_2$ (cf. \cite[Theorem~6.4.8]{KaltonAlbiac}).
\end{proof}

\begin{remark}
    The conclusion of Theorem~\ref{thm: disp_cont_equiball} extends to certain r.i. spaces on $[0,1]$ by adapting the arguments of \cite[Lemma~5.4]{Haydon11}. Specifically, it holds when $E$ is order continuous, has upper Boyd index $q_E<\infty$, and satisfies an upper $p$-estimate for some $p>2$.
\end{remark}

The following construction, inspired by \cite[Example~2]{Astashkin14}, shows that the conclusion of Theorem~\ref{thm: disp_cont_equiball} fails dramatically in $L_2$.

\begin{theorem}{\label{thm: subspace_L_2}}
    For every $0 < \epsilon < 1/2$ there exists an infinite dimensional dispersed subspace $M_{\epsilon} \on L_2$ such that
    \begin{equation}
        \inf \bigg\{ \liminf_{n \to \infty} d(x_n, M_{\epsilon}): (x_n) \on S_{L_2} \, \, \text{weakly null equi-integrable sequence}\bigg\} = 1- \epsilon.
    \end{equation}
    Hence, $M_{\epsilon}$ contains no infinite dimensional subspace with equi-integrable unit ball.
\end{theorem}

\begin{proof}
    Fix an arbitrary $0 < \epsilon < 1/2$, and take $\beta > 1/3$ such that $\beta(1+\beta)^{-1} = (1-\epsilon)^2$. For simplicity, we work on the probability space $\Omega=[0,1]\times[0,1]$, endowed with product Lebesgue measure. It is well known (cf. \cite[p. 102]{KaltonAlbiac}) that there exists a measure-preserving bijection $\phi\colon [0,1]\longrightarrow\Omega$. Consequently, the operator $U\colon L_2(\Omega)\longrightarrow L_2$ defined by $Uf=f\circ~\phi$ is a surjective lattice isometry, and thus the following construction in $L_2(\Omega)$ transfers lattice isometrically to $L_2$.

    Let $(r_n)$ be the Rademacher system. For $n \geq 1$, set $E_n = (2^{-n-1}, 2^{-n}]$, $\alpha_n = (\beta 2^{n+1})^{1/2}$ and define
    \begin{equation}
        g_n(t)= \begin{cases} \alpha_n ,&  t \in E_n,\\
            1,& t \notin E_n.
        \end{cases}
    \end{equation}
    Note that $m(E_n)=2^{-n-1}$ and $\alpha_n > 1$. Finally, put $f_n(s,t) = r_n(s)g_n(t)$ and $M_{\epsilon} = [f_n]$. We verify the required properties in a sequence of claims.
    
    \medskip   
    \noindent \textbf{Claim 1.} \textit{$(f_n)$ is an orthogonal basic sequence equivalent to the unit vector basis of $\ell_2$.} 
    
    For $j \neq k$, Fubini's theorem yields
    \begin{equation}
        \langle f_j, f_k \rangle = \int_0^1 g_j(t)g_k(t) \, dt \int_0^1 r_j(s)r_k(s) \, ds = 0
    \end{equation}
    Thus $(f_n)$ is an orthogonal sequence, and
    \begin{equation}{\label{ec: norm_block_f_k}}
        \bigg\Vert \sum_{k=1}^{\infty} a_k f_k \bigg \Vert_2^2 = \sum_{k=1}^{\infty} |a_k|^2\Vert f_k \Vert_2^2 = \sum_{k=1}^{\infty} |a_k|^2 ( \alpha_k^2 |E_k|+1-|E_k|) = \sum_{k=1}^{\infty} |a_k|^2 ( 1+\beta - 2^{-k-1}).
    \end{equation}
    Since $2^{-n-1} < \beta$, it follows that
    \begin{equation}{\label{ec: equiv_ell_2_f_k}}
        \sum_{k=1}^{\infty} |a_k|^2 \leq \bigg\Vert \sum_{k=1}^{\infty} a_k f_k \bigg \Vert_2^2 \leq (1+\beta)\sum_{k=1}^{\infty} |a_k|^2.
    \end{equation}
    Therefore $(f_n)$ is equivalent to the unit vector basis of $\ell_2$.
    
    \medskip
    \noindent \textbf{Claim 2.} \emph{The subspace $M_{\epsilon}$ is dispersed.} 

    By Khintchine's inequality (\cite[Theorem~6.2.3]{KaltonAlbiac}), there is a constant $C > 0$ such that 
    \begin{equation}
        \begin{aligned}
            \left\|\sum_k a_kf_k\right\|_1 &=
            \int_0^1\int_0^1 \left|\sum_k a_kr_k(s)g_k(t)\right|\,ds\,dt \\
            &\geq C\int_0^1 \left(\sum_k|a_k|^2g_k(t)^2\right)^{1/2}dt 
            \geq C\left(\sum_k |a_k|^2\right)^{1/2},
        \end{aligned}
    \end{equation}
    where we have used that $g_n \geq 1$. Combining this with
    the first claim, it follows that
    \begin{equation}
        \Vert u\Vert_1 \geq \frac{C}{\sqrt{1+\beta}}\,\Vert u\Vert_2, \quad u \in M_{\epsilon}.
    \end{equation}
    Hence $M_{\epsilon}$ is dispersed by Proposition~\ref{prop: dispersed_strongly_embedded} and \cite[Proposition~6.4.5]{KaltonAlbiac}.
\medskip

    \noindent\textbf{Claim 3.}
    \emph{For every nonempty finite $I\subseteq\mathbb N$ and every block
    $u=\sum_{k\in I}a_kf_k$, there is a measurable set $A_I \on \Omega$ such that
        $\Vert u\chi_{A_I} \Vert_2\geq (1-\epsilon) \Vert u\Vert_2$ and 
        $|A_I|\leq2^{-\min I}$.}
    \medskip
    
    Take $I\subseteq\mathbb N$ finite, $u=\sum_{k\in I}a_kf_k$ and 
    $A_I = [0,1] \times \cup_{k \in I} E_k$. For $k \in I$ we obtain
    \begin{equation}
        \begin{aligned}
            \int_{\bigcup_{j\in I}E_j}g_k(t)^2\,dt =
        \alpha_k^2|E_k|
        +
        \left|
            \left(\bigcup_{j\in I}E_j\right)\setminus E_k
        \right| \geq\beta.
        \end{aligned}
    \end{equation}
    Therefore, by the orthogonality of the Rademacher system,
    \begin{equation}
        \begin{aligned}
            \Vert u\chi_{A_I}\Vert_2^2
            &=
            \int_{\bigcup_{j\in I}E_j}\int_0^1
            \left|
            \sum_{k\in I}a_kr_k(s)g_k(t)
            \right|^2ds\,dt = \int_{\bigcup_{j\in I}E_j} \sum_{k\in I}|a_k|^2g_k(t)^2 \geq \beta \sum_{k\in I}|a_k|^2,
        \end{aligned}
    \end{equation}
    which, together with \eqref{ec: equiv_ell_2_f_k}, yields
    \begin{equation}
        \Vert u\chi_{A_I}\Vert_2^2 \geq \dfrac{\beta}{\beta+1}\Vert u\Vert_2^2 = (1-\epsilon)^2 \Vert u\Vert_2^2.
    \end{equation}
    Furthermore, $|A_I| = \sum_{k \in I}2^{-k-1} \leq 2^{-\min I}$.
    \medskip
    
    \noindent\textbf{Claim 4.}
    \emph{For every normalized weakly null equi-integrable sequence
    $(x_n)\subseteq L_2(\Omega)$, $\liminf_{n\to\infty}d(x_n,M_\varepsilon)\geq 1-\epsilon.$}

    Suppose, arguing by contradiction, that there exists a normalized weakly null equi-integrable sequence $(x_n)\subseteq L_2(\Omega)$ such that $d = \liminf_{n\to\infty}d(x_n,M_\varepsilon) < 1- \epsilon$. By passing to a subsequence we may assume that $d(x_n, M_{\epsilon}) \to d$. Let $P$ be the orthogonal projection onto $M_\varepsilon$ and put $y_n=Px_n$. Then
    \begin{equation}\label{eq:projection-norm-epsilon}
        \Vert x_n-y_n\Vert_2=d(x_n, M_{\epsilon}),
        \qquad
        \Vert y_n\Vert_2^2=1-d(x_n, M_{\epsilon})^2,
    \end{equation}
    and $(y_n)$ is weakly null. By the Bessaga-Pe\l czy\'nski principle, 
    passing to a further subsequence we can find a block sequence $u_n = \sum_{k=p_n}^{q_n} a_k f_k$ ($p_n < q_n$) such that $\Vert y_n - u_n \Vert_2 \to 0$. Now, by the third claim there exists a family of measurable sets $(A_n) \on \Omega$ such that $\Vert u_n \chi_{A_n} \Vert_2 \geq (1-\epsilon)\Vert u_n \Vert_2$ and $|A_n| \leq 2^{-p_n}$ for every $n \in \N$. Thus
    \begin{equation}
        \Vert y_{n}\chi_{A_n}\Vert_2
        \geq
        \Vert u_n\chi_{A_n}\Vert_2
        -
        \Vert(y_{n}-u_n)\chi_{A_n}\Vert_2\\
        \geq
        (1-\epsilon)\Vert u_n\Vert_2-\Vert y_{n}-u_n\Vert_2.
    \end{equation}
    Since $A_n$ and $A_n^c$ are disjoint for every $n \geq 1$, taking limits yields
    \begin{equation}{\label{ec: L_2_equi_n}}
        \begin{aligned}
            \limsup_{n \to \infty} \Vert y_{n}\chi_{A_n^c}\Vert_2^2 & = \limsup_{n \to \infty} \Big(\Vert y_{n}\Vert_2^2 - \Vert y_{n}\chi_{A_n}\Vert_2^2 \Big) \\
            & \leq \limsup_{n \to \infty} \Big(\Vert y_{n}\Vert_2^2 - \Big((1-\epsilon)\Vert u_n\Vert_2-\Vert y_{n}-u_n\Vert_2 \Big)^2 \Big)\\
            & = \limsup_{n \to \infty} \Big(\Vert y_{n}\Vert_2^2 - (1-\epsilon)^2\Vert u_n\Vert_2^2 \Big) = \Big(1-(1-\epsilon)^2\Big)(1-d^2).
        \end{aligned}
    \end{equation}
      
    On the other hand, since $p_n \to \infty$, we have $|A_n| \to 0$. Thus, since $(x_n)$ is equi-integrable, $\Vert x_n \chi_{A_n} \Vert_2 \to 0$. On the other hand $x_n-y_n \perp y_n$, so  $\Vert y_n \Vert_2^2 = \langle x_n, y_n \rangle$ and by splitting this identity over $A_n$ and $A_n^c$ and applying Cauchy–Schwarz, we get
    \begin{equation}
        \begin{aligned}
            \Vert y_{n}\Vert_2^2  = |\langle x_n, y_n \rangle| & \leq
        \Vert x_{n}\chi_{A_n}\Vert_2 \Vert y_{n}\chi_{A_n}\Vert_2 + \Vert x_{n}\chi_{A_n^c}\Vert_2 \Vert y_{n}\chi_{A_n^c}\Vert_2 \\
        & \leq
        \Vert x_{n}\chi_{A_n}\Vert_2 \Vert y_{n}\Vert_2 + \Vert y_{n}\chi_{A_n^c}\Vert_2.
        \end{aligned}
    \end{equation}
    Letting $n \to \infty$, we get $1-d^2 \leq \liminf_{n \to \infty} \Vert y_{n} \chi_{A_n^c} \Vert_2$, and combining this with \eqref{ec: L_2_equi_n} we obtain $(1-d^2)^2 \leq (1-(1-\epsilon)^2)(1-d^2) $. Thus, dividing by $1-d^2$ yields $d \geq 1-\epsilon$, the desired contradiction.
\medskip
    
    \noindent\textbf{Claim 5.}
    \emph{There is a normalized weakly null equi-integrable sequence
    $(x_n)\subseteq L_2(\Omega)$ such that $d(x_n,M_\varepsilon) \to 1-\epsilon.$}

    Consider the sequence $x_n(s,t)=r_n(s)$. Then $(x_n)$ is orthonormal, hence normalized and weakly null, and it
    is equi-integrable because $|x_n|=1$ almost everywhere. Moreover, $\langle x_n,f_k\rangle=0$ for $k\neq n$, while
    \begin{equation}
        \langle x_n,f_n\rangle =\int_0^1g_n(t)\,dt =\alpha_n|E_n| + 1 - |E_n| = \beta^{1/2} 2^{-\frac{n+1}{2}} + 1 - 2^{-n-1}.
    \end{equation}
    Thus
    \begin{equation}
        \begin{aligned}
        d(x_n,M_\varepsilon)^2 =1-\Vert Px_n\Vert_2^2 = 1-\dfrac{|\langle x_n,f_n\rangle|^2}{\Vert f_n \Vert_2} =
        1- \frac{(\beta^{1/2} 2^{-\frac{n+1}{2}} + 1 - 2^{-n-1})^2}{1+\beta-2^{-n-1}}.
        \end{aligned}
    \end{equation}
    Therefore $d(x_n,M_\varepsilon)\to \beta^{1/2}(1+\beta)^{-1/2} = 1- \epsilon$.

    \medskip
    \noindent\textbf{Claim 6.}
    \emph{The subspace $M_\varepsilon$ contains no infinite-dimensional subspace whose unit ball is equi-integrable.}

    Suppose that $M_\varepsilon$ contains an infinite-dimensional subspace $N$ whose unit ball is equi-integrable. An orthonormal sequence $(w_n)\subseteq N$ is normalized, weakly null, and equi-integrable. However, $d(w_n,M_\varepsilon)=0$ for $n \geq 1$, contradicting the fourth claim. 
\end{proof}

\section*{Acknowledgments}

The authors thank William B. Johnson for a helpful exchange concerning the range of validity of \cite[Lemma~1]{JohnsonSchechtman08}.

\input{References}
\end{document}

%% file: Preamble.tex
\usepackage[english]{babel}
\usepackage[T1]{fontenc}
\usepackage{accents}
\usepackage{amssymb}
\usepackage{amsthm}
\usepackage{amsmath}
\usepackage{mathtools} 
\usepackage{color}
\usepackage{enumerate}
\usepackage[shortlabels]{enumitem}
\usepackage{float}
\usepackage[export]{adjustbox}
\usepackage{hyperref}
\hypersetup{
	colorlinks=true,
	linkcolor=blue, %
	filecolor=magenta, 
	urlcolor=cyan,
	citecolor= blue  %
}
\usepackage[textheight=633pt, inner=2.5cm, outer=2.5cm, voffset= -1.3cm,tmargin=124pt]{geometry}
\usepackage{fancyhdr}
\usepackage[most]{tcolorbox}
\tcbuselibrary{theorems}

\usepackage{comment} 
\usepackage{etoolbox}

\DeclareMathOperator*{\DNS}{DNS}
\DeclareMathOperator*{\DSS}{DSS}
\DeclareMathOperator*{\coDNS}{coDNS}

\DeclareMathOperator*{\sol}{sol}
\DeclareMathOperator*{\con}{con}
\DeclareMathOperator*{\acon}{acon}

\newcommand{\N}{\mathbb{N}}

\newcommand{\R}{\mathbb{R}}

\newcommand{\on}{\subset}
\newcommand{\supp}{supp \, }

\newcommand{\Lc}{\mathcal B}

\theoremstyle{plain}
\newtheorem{theorem}{Theorem}[section]
\newtheorem{corollary}[theorem]{Corollary}
\newtheorem{lemma}[theorem]{Lemma}
\newtheorem{proposition}[theorem]{Proposition}
\newtheorem{question}{Question}

\makeatletter
\renewcommand*\l@subsection{\@tocline{2}{0pt}{2.8em}{2.8em}{}}
\makeatother

\theoremstyle{definition}
\newtheorem{definition}[theorem]{Definition}
\newtheorem{example}[theorem]{Example}
\newtheorem{remark}[theorem]{Remark}

\mathtoolsset{showonlyrefs}

\title{$L$-tauberian operators and subspaces with equi-integrable unit ball}

\author[A. Caldera]{Alberto Caldera}
\address{Departamento de Matem\'aticas, Facultad de Ciencias, Universidad de Oviedo, E-33007 Oviedo, Espa\~na}
\email{calderaalberto@uniovi.es}

\author[M. Gonz\'alez]{Manuel Gonz\'alez}
\address{Departamento de Matem\'aticas, Facultad de Ciencias, Universidad de Cantabria, E-39071 Santander, Espa\~na}
\email{manuel.gonzalez@unican.es}

\thanks{{\em 2020 Mathematics Subject Classification:} Primary 47B60; Secondary 47A53, 46B42, 46E30.\\
{\em Key Words:} Banach lattice, $L$-tauberian operator, equi-integrability, dispersed subspace.\\
The first author was partially supported by MICIU (Spain), Grant FPU24/03036.
}

%% file: main.bbl
\begin{thebibliography}{99}
			
	\bibitem{KaltonAlbiac} F. Albiac, N. J. Kalton, \textit{Topics in Banach Space Theory}, Graduate Texts in Mathematics \textbf{233}, 
    Springer, New York, 2006.

    \bibitem{AAG:92} J. A. Alvarez, T. Alvarez, M.   Gonz\'alez, \emph{The gap between subspaces and perturbation of non-semi-Fredholm operators,} Bull. Austral. Math. Soc. \textbf{45} (1992), 369–376.
    
	\bibitem{Astashkin14} S. V. Astashkin, \textit{$\Lambda(p)$-spaces}, J. Funct. Anal. \textbf{266} (2014), 5174-5198.

    \bibitem{Astashkin23} S. V. Astashkin, \textit{The structure of subspaces in Orlicz spaces lying between $L_1$ and $L_2$}, Math. Z. \textbf{303} (2023), 91.

	\bibitem{Bilokopytov22} E. Bilokopytov, \textit{Disjointly non-singular operators on order continuous Banach lattices complement the unbounded norm topology}, J. Math. Anal. Appl. \textbf{506} (2022), 125556, 13 pp.
		
	\bibitem{Bilokopytov25} E. Bilokopytov, \textit{Disjointly non-singular operators and various topologies on Banach lattices}, arXiv preprint arXiv:2501.03866 (2025).

    \bibitem{Caldera26} A. Caldera, M. Gonz\'alez, A. Mart\'inez-Abej\'on, \textit{Operators whose conjugate is disjointly non-singular}, preprint, 2026.
    
	\bibitem{Troitsky17} Y. Deng, M. O'Brien, V.G. Troitsky, \textit{Unbounded norm convergence in Banach lattices}, Positivity \textbf{21} (2017), 963-974.

    \bibitem{Dor75} L. E. Dor, \textit{On projections in $L_1$}, Ann. Math. \textbf{102} (1975), 463--474.

    \bibitem{Flores08} J. Flores, F. L. Hernández, P. Tradacete, \textit{Powers of operators dominated by strictly singular operators}, Q. J. Math. \textbf{59} (2008), 321-334.
    
	\bibitem{Freeman24} D. Freeman, T. Oikhberg, B. Pineau, M. A. Taylor, \textit{Stable phase retrieval in function spaces}, Math. Ann. \textbf{390} (2024), 1–43.
		 
	\bibitem{TauberianL1} M. González, A. Martínez-Abejón, \textit{Tauberian operators on $L_1(\mu)$ spaces}, Studia Math. \textbf{125} (1997), 289-303.

    \bibitem{TauberianOperators} M. González, A. Martínez-Abejón, \textit{Tauberian Operators}, Operator Theory: Advances and Applications 194, Birkhäuser, 2010.

    \bibitem{Manuel21} M. González, A. Martínez-Abejón, A. Martinón, \textit{Disjointly non-singular operators on Banach lattices}, J. Funct. Anal. \textbf{280} (2021), 108944.


    \bibitem{ManuelQuantitative21} M. González, A. Martinón, \textit{A quantitative approach to disjointly non-singular operators}, Rev. R. Acad. Cienc. Exactas Fís. Nat. Ser. A Mat. \textbf{115} (2021), 185.

 	\bibitem{Manuel24} M. González, A. Martinón, \textit{Disjointly non-singular operators: extensions and local variations}, J. Math. Anal. Appl. \textbf{530} (2024), 127685.
        
    \bibitem{GonzalezOnieva90} M. Gonz\'alez, V. M. Onieva, \textit{Characterizations of tauberian operators and other semigroups of operators}, Proc. Amer. Math. Soc. \textbf{108} (1990), 399--405.

    \bibitem{ManuelSaksman95} M. González, E. Saksman, H.-O. Tylli, \textit{Representing non-weakly compact operators}, Stud. Math. \textbf{113} (1995), 289–303.
    	
    \bibitem{Haydon11} R. Haydon, E. Odell, Th. Schlumprecht, \textit{Small subspaces of $L_p$}, Ann. Math. \textbf{173} (2011), 169-209.


    \bibitem{JohnsonSchechtman08} W. B. Johnson, G. Schechtman, \textit{Multiplication operators on $\mathcal{L}(L_p)$ and $\ell_p$-strictly singular operators}, J. Eur. Math. Soc. \textbf{10} (2008), 1105--1119.

    \bibitem{Kadec} M. I. Kadec, A. Pełczy\'nski, \textit{Bases, lacunary sequences and complemented subspaces in the spaces $L_p$.} Studia Math. \textbf{21} (1961), 161–176.
        
    \bibitem{KaltonWilansky76} N. Kalton, A. Wilansky, \textit{Tauberian operators on Banach spaces}, Proc. Am. Math. Soc. \textbf{57} (1976), 251–255.
    
    \bibitem{Kato} T. Kato, \textit{Perturbation Theory for Linear Operators}, corrected printing of the 2nd ed., Springer, 1995.

    \bibitem{Komlos} J. Koml\'os, \textit{A generalization of a problem of Steinhaus}, Acta Math. Hungar. \textbf{18} (1967) 217–229.
    
	\bibitem{Lesnik} K. Le\'snik, L. Maligranda, J. Tomaszewski, \textit{Weakly compact sets and weakly compact pointwise multipliers in Banach function lattices}, Math. Nachr. \textbf{295} (2022), 574-592.
		
                
	\bibitem{LindenstraussTzafiri} J. Lindenstrauss, L. Tzafriri, \textit{Classical Banach Spaces II, Function Spaces}, Springer, 1979.

    \bibitem{Meyer-Nieberg:91} P. Meyer-Nieberg, \textit{Banach lattices}, Springer, 1991.
	
	\bibitem{Wnuk} W. Wnuk, \textit{Banach Lattices with Order Continuous Norms}, Polish Scientific Publishers PWN, 1999.

		
\end{thebibliography}
